\documentclass[preprint,10pt,1p]{elsarticle}
\usepackage{enumerate,color,amssymb,amsmath,ntheorem}

\usepackage{gnuplottex}
\usepackage{tikz}
\usepackage[toc]{appendix}

        \usepackage{dutchcal}

            \newtheorem{thm}{Theorem}[section]
          \newtheorem{prop}[thm]{Proposition}
          \newtheorem{lem}[thm]{Lemma}

          \newtheorem{rem}[thm]{Remark}

\usepackage{etoolbox}
\usepackage{amsmath,amsfonts,enumerate,amsbsy}
\newcommand{\dom}{{\partial \Omega}}
\newcommand{\bfx}{{\boldsymbol x}}

\DeclareMathOperator*{\argmin}{arg\,min}

\newproof{proof}{Proof}
\usepackage{hyperref}

\usepackage{geometry}
\AtEndEnvironment{proof}{\qed}

\def\ov{\overline}

\def\D{{\mathcal D}}

\newcommand{\RR}{\mathbb R}

\newcommand{\N}{\mathbb N}

\newcommand{\rr}{{\RR_+}}

\newcommand{\ue}{\frac{1}{\varepsilon}}

\newcommand{\dt}{\partial_t}

\newcommand{\e}{\varepsilon}

\newcommand{\normel}{\left\|}
\newcommand{\normer}{\right\|}

\newcommand{\sgn}{{\rm sgn}}

\newcommand{\nud}{{\frac{1}{2}}}

\newcommand{\ddt}[1]{\frac{{\rm d} #1}{\rm d \it t}}
\newcommand{\ti}[1]{\tilde{#1}}

\newtheorem{coro}{Corollary}[section]
\newtheorem{defi}{Definition}[section]
\newtheorem{hypo}{Assumptions}[section]

\newcommand{\nrm}[2]{{\normel{#1}\normer}_{#2}}

\newcommand{\SysNoNumB}[1]{
\begin{displaymath}
\left\{
\begin{array}{#1}
}
\newcommand{\SysNoNumE}{
\end{array}
\right.
\end{displaymath}
}
\newcommand{\TabNoNumB}[1]{
\begin{displaymath}
\begin{array}{#1}
}
\newcommand{\TabNoNumE}{
\end{array}
\end{displaymath}
}
\newcommand{\TabNumB}[2]{
\begin{equation}
\label{#2}
\begin{array}{#1}
}
\newcommand{\TabNumE}{
\end{array}
\end{equation}
}
\newcommand{\SysNumB}[2]{
\begin{equation}
\left\{
\label{#2}
\begin{array}{#1}
}
\newcommand{\SysNumE}{
\end{array}
\right.
\end{equation}
}

\newcommand{\ddn}[1]{\,\partial_{\nu} #1\,}

\newcommand{\cM}{\mathcal{M}}

\newcommand{\etau}[1]{}

\newcommand{\chiu}[1]{\ensuremath{\chi_{#1}}}

\newcommand{\cR}{{\mathcal R}}

\newcommand{\rhoe}{\rho_\e}

\newcommand{\rhoi}{\rho_{I}}
\newcommand{\rhoz}{\rho_0}

\newcommand{\ztlip}{\zeta_{\rm Lip}}

\newcommand{\da}{\partial_a} 
\newcommand{\zeps}{z_\e}
\newcommand{\mueps}{\mu_{0,\varepsilon}}

\newcommand{\veps}{u_\e}
\newcommand{\vepsi}{u_{I}}

\newcommand{\zp}{z_p}

\newcommand{\bmax}{\beta_{M}}
\newcommand{\bmin}{\beta_{m}}
\newcommand{\ztmax}{\zeta_{M}}
\newcommand{\ztmin}{\zeta_{m}}
\newcommand{\mumin}{\mu_{0,m}}

\newcommand{\beps}{\beta_{\varepsilon}}
\newcommand{\bz}{\beta_0}
\newcommand{\zteps}{\zeta_\e}
\newcommand{\ztz}{\zeta_0}
\newcommand{\hrhoe}{\hat{\rho}_\varepsilon}

\newcommand{\muoz}{\mu_{1,0}}
\newcommand{\muzz}{\mu_{0,0}}
\newcommand{\muzm}{\mu_{0,\min}}

\newcommand{\muze}{\mu_{0,\e}}

\newcommand{\heps}{h_{\e}}

\newcommand{\cS}{{\mathcal S}}

\newcommand{\hrhoi}{\hat{\rho}_{\e,I}}

\newcommand{\hatw}{\hat{w}}

\newcommand{\hatmu}{\hat{\mu}}

\newcommand{\st}{\text{ s.t. }}

\newcommand{\hmu}{\hat{\mu}}

\newcommand{\tia}{{\ti{a}}}

\newcommand{\bhoz}{H^1_0(\om)}

\newcommand{\vepsp}{\veps^+}

\newcommand{\cE}{{\mathcal E}}

\newcommand{\rhow}{\varrho_w}

\newcommand{\dtbmax}{|\dt \bz|_{\infty}}
\newcommand{\dtztmax}{|\dt \ztz|_{\infty}}

\newcommand{\hz}{{\hat{z}_\e}}

\newcommand{\alev}{\;{\rm a.e}\;\;}

\newcommand{\qt}{Q_T}

\newcommand{\om}{{\Omega}}
 \newcommand{\tet}{{D_t^\tau}}

\newcommand{\xat}{\bfx,a,t}

\newcommand{\cL}{{\cal L}}

\begin{document}

\begin{frontmatter}


\title{ Space dependent adhesion forces \\
mediated by transient elastic linkages~: \\
new convergence and global existence results
}

\author{Vuk Mili\v si\' c}
\address{Centre National de la Recherche Scientifique, \\
Laboratoire Analyse G{\'e}om{\'e}trie \& Applications,\\
Universit{\'e} Paris 13,  Sorbonne Paris Cit\'e, France}
\ead{milisic@math.univ-paris13.fr}
\author{Dietmar Oelz}
\address{
School of Mathematics and Physics, 
University of Queensland, Australia} 
\ead{d.oelz@uq.edu.au}



\begin{abstract}
In the first part of this work we show the convergence with respect to an asymptotic parameter $\e$ of a delayed heat equation. 
It represents a mathematical extension of works considered previously by the authors \cite{MiOel.1,MiOel.2,MiOel.3}.
Namely, this is the first result involving delay operators approximating protein linkages coupled with 
a spatial elliptic second order operator.
For the sake of simplicity we choose the Laplace operator, although more general results could be derived.
The main arguments are (i) new energy estimates and (ii) a  stability result 
extended from the previous work to this more involved context. 
They allow to prove  convergence of the delay operator to 
a friction term together with the Laplace operator in the same asymptotic regime considered without the space dependence  in~\cite{MiOel.1}.
In a second part we extend fixed-point results for the fully non-linear model introduced in \cite{MiOel.3} and prove global existence 
in time. This shows that the blow-up scenario observed previously 
does not occur. Since the latter result was interpreted as a rupture of adhesion forces, 
we discuss the possibility of bond breaking both from the analytic and numerical point of view.
\end{abstract}

\end{frontmatter}

\setcounter{tocdepth}{2}

\section{Introduction}

\subsection{Biological and mathematical settings}
Cell migration is an ubiquitous process underlying morphogenesis, wound healing and cancer, among other biological phenomena \cite{bray2001cell}.
Leading-edge protrusion on flat surfaces - the first step in cell crawling -  relies on continuous remodeling of a cytoskeletal structure called the lamellipodium \cite{pmid12600310}, a broad and flat network of actin filaments.

Comprehensive modeling efforts were initiated in 1996 and fall into two groups. The first group includes continuum models for the mechanical behaviour of cytoplasm \cite{pmid10204394, 3a67dc837b5f4ea3b74a3b7f102b2341}.
The second group makes assumptions about the microscopic organization of the actin network \cite{pmid8968574,pmid17477841}. 
In an attempt to create a framework that addresses the interplay of macroscopic features of cell migration and the meshwork structure, the Filament Based Lamellipodium Model has been developed. It is a two-dimensional, two-phase, anisotropic continuum model for the dynamics of the lamellipodium network which retains key directional information on the filamentous substructure of this meshwork
\cite{OeSchVi, MR3385931, MR3675606}. 

The model has been derived from a microscopic description based on the dynamics and interaction of individual filaments \cite{OeSch}, and it has by recent extensions \cite{MR3675606} reached a certain state of maturity. The main unknowns of the model are the positions of the actin filaments in two locally parallel families. 
The filaments are submitted to various forces : bending, twisting, in-extensibility, pressure, stretching and adhesion. These two latter
mechanisms, that stabilize the whole filament network, are at the heart of our project. In \cite{OeSch}, a formal derivation leaded to the expression of these forces as
operators depending on {\em friction} terms in the equations denoted instantaneous cross-link/adhesion turnover.
The dimensionless parameter $\e$ is the ratio between the reference value for the age of adhesions and the maximal life time of a monomer as part of a filament. 
This parameter is assumed to be small and the fact that the elasticity is $O(\e^{-1})$, is a scaling assumption required for a non-vanishing effect of adhesions in the limit $\e \to 0$. Our works construct various tools in order to handle rigorously this asymptotic \cite{MiOel.1,MiOel.2, MiOel.3,mi.proc}. 
In addition, concerning adhesion forces, a similar Ansatz was performed formally in a somehow different mechanical setting in \cite{PreVi,grec:hal-01566770}.

In  previous works we handled a single point adhesion with respect to the space variable. 
Indeed our unknown was the position of a unique point in time $\zeps(t)$.
In \cite{MiOel.1} we gave the first result of convergence based on a special Lyapunov 
functional for the linkages population and a comparison principle generalizing Gronwall's Lemma in the case of integral positive operators. 
In a second step \cite{MiOel.2} we found a new formulation of the problem, weakened some of the hypotheses of the first paper
and gave a fixed point theorem for a fully non-linear version of our new model. 
Here we give a comprehensive 
extension of convergence results in the weakly coupled setting (see below for a precise explanation) 
in the case of space dependent adhesion forces coupled with a second order elliptic 
operator.  
For the sake of simplicity it is chosen to be the Laplacian, but results presented hereafter could be extended to a broad class of  linear div-grad 
operators with  Dirichlet boundary conditions. 
To our knowledge, this is the first extension made in this direction starting from the initial single point model in \cite{MiOel.1}.

In \cite{MiOel.3} we considered a fully non-linear coupling for which the death rate of bonds depends on the positions of adhesions. There, we have shown that there could be a blow-up
in finite time for well-prepared data. Biologically this could be interpreted as tear-off of bonds, a detachment observed in experiences (cf. \cite{PreVi} and references therein). Here the presence of 
another term in the force balance prevents the blow-up, global existence in time is obtained without restrictions
on the data. If moreover $\beps$, the birth rate of the bond population admits a strictly positive lower bound $\beps\geq \bmin>0$, then 
one shows that this population actually never becomes extinct and an asymptotic profile is computed.
We underline that this latter hypothesis is crucial in many of our theoretical results.
In a last step we confront these results with a numerical simulation contradicting this latter hypothesis
and show that detachment can occur on compact sets inside the domain.

\subsection{A detailed mathematical framework}
\newcommand{\cet}{{\cal E}_t}
$\Omega$ denotes an open bounded connected set of $\RR^n$, whose boundary $\partial \Omega$ is $C^{1,1}$ (see for instance Definition 1.2.1.1. \cite{Grisvard}). 
For any fixed time $T$, the parabolic cylinder is denoted $Q_T:= \Omega\times (0,T)$.
The position of the moving binding site, $\zeps(\bfx,t)$,  minimizes at each time $t\geq 0$ an energy functional~:
\begin{equation}\label{eq.minimiz}
 \zeps(\bfx,t) = \argmin_{w \in H^1_0(\Omega)} {\cal E}(w) ,
\end{equation}
the energy being defined for every $w\in H^1_0(\om)$ as 
\begin{equation}\label{eq.nrj}
 \cet(w(\cdot)) :=\nud  \int_{\Omega} \left\{  | \nabla w |^2 + \int_{\rr} \frac{| w(\bfx)- \zeps(\bfx,t-\e a) |^2 }{\e} \rhoe(\bfx,t,a) da \right\} d\bfx \; ,
\end{equation}
the second term is  a delay operator since the minimisation is performed with respect to past positions $\zeps(\bfx,t-\e a)$. 
When $t<0$, these are given by the function $\zeps(\bfx,t)=\zp(\bfx,t)$ for $t<0$. 
The
age distribution $\rhoe=\rhoe(\bfx,a,t)$ is the solution of the structured model~:
\begin{equation}
\label{eq.rho.eps} 
\left\{ 
\begin{aligned} 
&\varepsilon \partial_t \rhoe + \partial_a \rhoe+ \zteps \, \rhoe = 0 
\,, &\bfx \in \Omega, \, a>0 \, , \;t > 0 , \\ 
& \rhoe(\bfx,a=0,t)=\beps(\bfx,t)\left(1-\muze(\bfx,t)  \right) 
\, , &\bfx \in \Omega,\; a=0,\, t > 0  , \\
& \rhoe(\bfx,a,t=0)=\rhoi(\bfx,a)
\, ,& \bfx \in \Omega, a>0, t=0,
\end{aligned}  
\right. 
\end{equation}
where $\muze(\bfx,t):=\int_0^\infty \rhoe(\bfx,\tia,t) \, d \tia$ and the on-rate of bonds is a given
function $\beps$ times a factor, that takes into account saturation of the moving binding site with linkages.
When the off-rate $\zteps$ is a prescribed function, we say that the problem is weakly coupled~:
first one solves $\rhoe$ and then $\rhoe$ is  the  integral term in \eqref{eq.nrj}  providing $\zeps$.
If instead $\zeta$ is a function depending on $\zeps$, or which is more biologicaly sound (cf. \cite{Suda_2001, Li2003}),  on the 
elongation $\veps(\bfx,a,t):=(\zeps(\bfx,t)-\zeps(\bfx,t-\e a))/\e$ the problem is said to be fully coupled.

Note that the Euler-Lagrange equation associated to the minimization process is a 
Volterra equation of the first kind \cite{MiOel.1} given by
\begin{equation}
\label{eq.z.eps} 
\left\{ 
\begin{aligned} 
&\cL_\e(\zeps,\rhoe)  = \Delta_{\bfx} \zeps \; , \quad & t\geq 0,\; \bfx \in \Omega \; , \\
& \zeps(\bfx,t) = 0 ,& \;t \in \rr  \; ,\; \bfx \in \dom ,\\
& \zeps(\bfx,t)=\zp(\bfx,t) \; , \quad  &t < 0 \; ,\; \bfx \in \Omega ,
\end{aligned}  
\right. 
\end{equation}
where $\cL_\e(\zeps,\rhoe)(\bfx,t) := \ue \int_\rr \left( \zeps(\bfx,t)-\zeps(\bfx,t-\e a)\right)\rhoe(\bfx,a,t) da $.
It is easy to prove that if $\zeps$ solves \eqref{eq.z.eps} in the variational sense for every time $t\geq0$, then it minimizes \eqref{eq.nrj} and vice-versa (cf Appendix \ref{sec.equiv}).

In contrast to the previous results reported in \cite{MiOel.1,MiOel.2,MiOel.3}, we introduce 
the space dependence through the $\bfx$ variable and through  a partial differential operator on the right hand side of \eqref{eq.z.eps}.
In a first part, we show rigorously, in the semi-coupled case, 
that indeed the solutions of \eqref{eq.rho.eps}-\eqref{eq.z.eps} converge, as $\e$ goes to 0, to the solutions of  the limit 
equations~:
\newcommand{\zz}{{z_0}}
\begin{equation}\label{eq.z.0}
\left\{
 \begin{aligned}
& \muoz(\bfx,t) \dt \zz - \Delta_\bfx \zz =0, & (\bfx,t) \in \Omega \times \rr, \\
&  \zz(\bfx,t)=0 , & (\bfx,t)\in\dom \times \rr,   \\
& \zz(\bfx,0) = \zp(\bfx,0), & (\bfx,t)\in \om\times\{0\}.
\end{aligned}
\right.
\end{equation}
The first equation above is to be understood in the $L^2(\qt)$ sense.
The function $\mu_{k,0} := \int_{\rr} a^k \rhoz (\bfx,a,t) \, da$ represents the moment of order $k$ of $\lim_{\e\to 0} \rhoe =: \rhoz$ which solves 
\begin{equation}\label{eq.rho.zero}
\left\{
 \begin{aligned}
 &  \partial_a \rhoz+ \ztz \, \rhoz = 0 
\,, &\bfx \in \Omega , \;a>0, \; t>0, \\ 
& \rhoz(\bfx,a=0,t)=\bz(\bfx,t)\left(1-\muzz(\bfx,t)  \right) 
\, , & \bfx \in \Omega , \;a=0, \; t>0.\\
\end{aligned}
\right.
\end{equation}
These convergence results are  essentially due to two new ingredients~: 
\begin{enumerate}[(i)]
 \item  we prove a new energy estimate (see Theorem \ref{thm.nrj})
which states that $\cet(\zeps(\cdot,t)) \leq \cE_0(\zeps(\cdot,0))$  providing a first compactness result.
Since delay terms often induce oscillations in time, 
this key result shows they are controlled by the energy minimized at each step. 
A similar result  is provided when adding a source term $\cS$ in section~\ref{sec.b}. 
\item   considering the elongation variable  introduced in \cite{MiOel.2}, we prove a 
stability result, which is mathematically more involved than in our previous papers (cf Theorem \ref{thm.stab.rho.veps} versus estimates (2.6) p.6 in \cite{MiOel.2}). The main difficulty is caused by the presence of the Laplace operator. Instead in 
the previous articles a given source term $\cS(t)$ (independent on $\zeps$) was prescribed
and greatly simplified these stability estimates.
 This second step provides  a stronger control in time on the delay part of the energy $\cet$ but requires stronger hypotheses 
 on the data as well (see assumptions \ref{hypo.data.trois} $i)b)$).
\end{enumerate}

In a second step we consider, for a fixed $\e$, existence and uniqueness of the fully coupled problem
where $\zeta$ is a Lipschitz function of $\veps$, the elongation. 
In \cite{MiOel.3}, this model was 
considered at a single point. 
Here the presence of the space variable greatly complexifies
the mathematical setting. Nevertheless, we prove that there is global existence with no
specific restrictions on the data. 
This result is to be compared with \cite{MiOel.3}, where a blow-up could be shown 
under certain conditions on the data. Instead,  the presence of the Laplace operator
 precludes a singular limit of the delay term $\cL_\e$ for which $\rhoe \to 0$ and $\zeps(\bfx,t)-\zeps(\bfx,t-\e a)$ explodes 
 when the source term becomes too large.
 We show, as well, that if $\beps \geq \bmin >0$, there is no extinction of the total population $\muze$
which demonstrates that however great is  the external load $\cS$, no tear-off occurs
and  new bonds are constantly created at local positions $\zeps(\bfx,t)$.
We show as well that positivity of the elongation is preserved.
As in \cite{MiOel.3}, in the case where $\zeta(u)=1+|u|$, an autonomous equation on the total population of bonds is shown:
$$
\e \dt \muze + (\beps+1) \muze + \Delta \zeps + \cS = \beps, \quad \alev \; (\bfx,t)\in \om\times(0,T),
$$
giving  an asymptotic profile for large times.
Numerical simulations illustrate these latter comments and show two possible regimes
according to whether $\beps$ locally vanishes or not~: if for some $\bfx_0$ and $t>t_0$ $\beps(\bfx_0,t)=0$, 
then $\muze(\bfx_0,t) \to 0$ when $t \to \infty$ which biologically means detachment, or
$\beps(\bfx,t)\to \beta_\infty(\bfx)>0$ and then $\muze \to \beta_\infty(\bfx)/(\beta_\infty(\bfx)+1)$ which represents a steady adhesion.

In section  \ref{sec.hypo}, we give notations and  hypotheses useful throughout the paper. In section \ref{sec.exist},
we set up for fixed $\e$ the material necessary to guarantee existence, uniqueness and the correct functional 
spaces to which our solutions $(\rhoe,\zeps)$ belong, in a way not necessarily uniform with respect to $\e$.
In section \ref{sec.nrj},  we give a new energy inequality, stating that the energy $\cet$ minimized 
at each time, actually decreases. 
Then, in the same section, we provide a stability result already presented in our previous works but adapted 
to this more complicated  framework. 
In section \ref{sec.conv}, we assemble, in Theorem \ref{thm.cvg},  previous results and provide a rigorous proof of the 
convergence of $(\rhoe,\zeps)$  towards the solutions of \eqref{eq.z.0}-\eqref{eq.rho.zero}. 
Section \ref{sec.b} extends previous results when a given source term is added to \eqref{eq.z.eps}. 
In section \ref{sec.full}, we show global existence, uniqueness and positivity, for the fully coupled model. Numerical simulations
illustrate these results in the same section.


\section{Notations and hypotheses}\label{sec.hypo}

\noindent In the rest of the article the subscripts $\bfx$, $a$ or $t$ denote the functional spaces associated with the corresponding variables. For instance $L^\infty_{\bfx,a,t} := L^\infty(\om\times\rr\times(0,T))$ whereas $W^{1,\infty}_tL^2_\bfx := W^{1,\infty}([0,T];L^2(\om))$.

\begin{hypo}\label{hypo.data}
The dimensionless parameter $\varepsilon > 0$ is assumed to induce two families of chemical rate functions $\zteps \in L^\infty_{\bfx,a,t}$ and $\beps \in L^\infty_{\bfx,t}$ that satisfy~:
\begin{enumerate}[(i)]
\item For limit functions $\beta_{0} \in W^{1,\infty}
(Q_T)$ and
$\zeta_{0} \in W^{1,\infty}(\Omega\times\rr\times[0,T])$ it holds that 
\begin{equation*}
\nrm{ \zteps-\ztz}{L^{\infty}_{\bfx,a,t}} \to 0 \quad \text{and} \quad \nrm{ \beps - \bz}{L^\infty_{\bfx,t}}  \to 0
\; 
\end{equation*}
as ${\varepsilon} \to 0$.
\item We also assume that there are upper and lower bounds such that
\begin{equation*}
0 < \ztmin \leq \zteps(\bfx,a,t) \leq \ztmax \quad \text{and} \quad 
0 < \bmin \leq \beps(\bfx,t) \leq \bmax
\end{equation*}
for all $\varepsilon>0$, $\bfx \in \Omega$, $a \geq 0$ and $t>0$. 
\end{enumerate}
\end{hypo}

\noindent The initial data for the density model \eqref{eq.rho.eps} satisfies 
\begin{hypo}\label{hypo.data.deux}
  The initial condition $\rhoi \in L^\infty
  (\Omega\times\rr)$ satisfies
\begin{itemize}
\item positivity and boundedness~: there exists $M > \bmax$, s.t. 
$$
M\geq \rhoi(\bfx,a)\geq0 \; ,\quad \text{ a.e. in }  \Omega\times\rr\; ,
$$
moreover, one has also that the total initial population satisfies
$$
0< \int_{\rr} \rhoi(\bfx,a) da < 1
$$
for almost every $\bfx \in \Omega$.
\item boundedness of higher moments,
$$
0< \mu_{p,I}:=\int_{\rr} a^p \rhoi (\bfx,a) \; da \leq  c_p \; ,\quad \text{for } p\in\{1,2\}   \; ,
$$
where $c_p$ are positive constants depending only on $p$.
\item initial integrability~: 
$$
\int_{\rr} \sup_{\bfx \in \Omega} \left| \rhoi(\bfx,a) \right| a^p da < \infty,\quad \text{for } p\in\{0,1,2\}   \; .
$$
\end{itemize}

\end{hypo}
Concerning the integral equation \eqref{eq.z.eps} we assume
\begin{hypo}\label{hypo.data.trois}
 The past data satisfies~:
\begin{enumerate}[i)]
\item at time $t=0$ we assume that  
\begin{enumerate}[a)]
 \item $\zp(\cdot,0)$ is in $H^1_0(\Omega)$, 
 \item $\Delta \zp(\cdot,0) \in L^1(\Omega)$.
\end{enumerate}
 \item When $t\leq 0$ one assumes furthermore that~:
 $$
 \zp \in C(\RR_-;L^2(\Omega)) ,\quad \dt \zp \in L^\infty(\RR_-,L^2(\Omega)).
 $$
 where $C(\RR_-;L^2(\Omega))$ denotes continuous $L^2$-valued functions endowed with the
 uniform continuity semi-norms.
 The latter hypotheses translate into a Lipschitz constant which is $L^2$ in space~:
 $$
 | \zp(\bfx,t_2) - \zp(\bfx,t_1) | \leq C_{\zp}(\bfx)|t_2-t_1|, \quad \forall  (t_2,t_1) \in (\RR_-)^2\;,
 $$
 where $C_{\zp}(\bfx) \in L^2(\Omega)$.
\end{enumerate}
\end{hypo}

\begin{rem}
 Most of the hypotheses presented here are set for general convenience {\em i.e.} 
 in order to give the broader possible sense to  mathematical results claimed
 hereafter. In the biological context, the data are simply measured
 microscopic constants  (see for instance
 tables given in \cite{MR3675606, OeSch, OeSchVi}).
\end{rem}

\section{Existence and uniqueness results}\label{sec.exist}
\subsection{Extension of previous results for $\rhoe$}

\noindent For the problem solved by $\rhoe$, $\bfx$ is only a mute parameter and the theory 
established in \cite{MiOel.1}, holds for a.e.  $\bfx \in \Omega$. 
\begin{thm}
\label{prop.rho.exist} Let  assumptions~\ref{hypo.data} and \ref{hypo.data.deux} hold, 
then for every fixed $\varepsilon$ there exists a unique solution 
$\rhoe\in C_t(\rr; L^\infty_\bfx(\Omega;L^1_a(\rr))) \cap L^\infty(\Omega\times(\rr)^2)$ of the problem \eqref{eq.rho.eps}. It satisfies 
\eqref{eq.rho.eps} in the sense of characteristics, namely
\begin{equation}
\label{rho_model_by_characteristics}
\rhoe(\bfx,a,t)=\begin{cases}
\beta_\varepsilon(\bfx,t-\varepsilon a)\left(1- \int_{\rr} \rhoe(\bfx,\tilde a,t-\varepsilon a) \, d\tilde a\right)\times  & \\
\quad\quad \quad \quad  \times \exp \left(-\int_0^{a}  \zeta_\varepsilon(\bfx,\tilde a, t- \varepsilon (a-\tilde a))\; d \tilde a  \right) \; ,  & a < t/\varepsilon  \; , \\
\rhoi(\bfx,a-t/\varepsilon) \exp \left(-\frac{1}{\varepsilon} \int_0^{t} \zeta_\varepsilon(\bfx,(\tilde t-t )/\varepsilon + a, \tilde t)  \; d \tilde t   \right)  \; ,  &a \geq t/\varepsilon \; .         
        \end{cases}
\end{equation}
\end{thm}
  
\noindent We recall the Lemma 2.1 \cite{MiOel.1} that we adapt here adding the $\bfx$ contribution~:
\begin{lem}\label{lemxyz} 
Let $\rhoe$ be the unique solution of problem \eqref{eq.rho.eps} according to Theorem~\ref{prop.rho.exist}, then it satisfies a weak formulation 
\begin{multline}
\label{equ_rho_weak}
\int_{Q_T\times \rr} \rhoe(\bfx,a,t)  \left( \e \dt \varphi+\da \varphi+ \zteps \varphi\right) \; d \bfx \; dt \; da \;  - \e \int_{\Omega\times\rr} \rhoe(\bfx,a,t) \varphi(a,t=T) \; da \, d \bfx \; + \\
 + \int_{Q_T}  \rhoe(\bfx,a=0,t) \, \varphi (\bfx,0,t) \; dt \, d\bfx + \e \int_{\Omega\times \rr} \rhoi(\bfx,a) \varphi(\bfx,a,t=0) \; da = 0 \; , 
\end{multline}
for every $T>0$ and every test function $\varphi \in C^\infty(Q_T\times\rr) \cap L^\infty(Q_T\times\rr)$.
\end{lem}

\begin{lem}
\label{prop.rho.bounds} 
Let assumptions \ref{hypo.data} and \ref{hypo.data.deux} hold, 
then the unique solution 
$\rhoe\in C(\rr;$ $L^\infty(\Omega;L^1(\rr))) \cap L^\infty\left(\Omega\times \left(\rr\right)^2\right)$ of the problem \eqref{eq.rho.eps} from
Theorem~\ref{prop.rho.exist} satisfies
$$\rhoe(\bfx,a,t) \geq 0 \quad \text{a.e. in} \quad \Omega\times \rr^2 \quad \text{and}$$
\begin{equation}
\label{def_mu_min} 
\muzm \leq  \mueps(\bfx,t) <  1 \; , \quad \forall t \in \rr 
\quad \text{where} \quad
\muzm := \min \left( \mueps(0), \frac{\bmin}{\bmin + \ztmax} \right)
\; .
\end{equation}
\end{lem}

\begin{lem}\label{lem.rhoz.est}
 Under the hypotheses on $\ztz$ and $\bz$ in assumptions \ref{hypo.data}, one has~: 
$$
 \rhoz(\bfx,a,t)  \leq C  \exp( - \ztmin a), \quad 
  | \dt \rhoz (\bfx,a,t)|  \leq  C (1+a) \exp(-\ztmin a) 
$$
where the generic constants depend only on $(\bmax$, $\bmin$, $|\dt \bz|_\infty$, $\ztmax$, $\ztmin$, $|\dt \ztz|_\infty)$.
\end{lem}

\begin{proof}
One solves \eqref{eq.rho.zero}
 $$
\begin{aligned}
 \rhoz(\bfx,a,t)  \leq &\frac{\bmax \ztmax}{\ztmax + \bmin} \exp( - \ztmin a), \\
  | \dt \rhoz (\bfx,a,t)|  \leq & \left( \frac{\dtbmax\ztmax }{\ztmax +\bmin} + \frac{ \bmax \ztmax^2 }{(\ztmax + \bmin)^2} \left(\frac{ \dtbmax }{\ztmin} + \bmax \frac{\dtztmax}{\ztmin^2} \right) \right) \exp(-\ztmin a) +\\
  &\quad  + \frac{\bmax \ztmax}{\ztmax + \bmin} \dtztmax a \exp( - \ztmin a)
  \end{aligned}
$$

\end{proof}

\subsection{Characterizing $\zeps$, the solution of \eqref{eq.z.eps}}

We define the space where $\zeps$ shall evolve setting
$
X_T := L^\infty((0,T);H^1_0(\Omega))
$
for every positive real $T$.
\begin{defi}\label{def.weak.sol.z.eps}
We say that $\zeps$ solves \eqref{eq.z.eps} in the  weak sense for  $\e$ fixed, if $\zeps \in X_T$ and if it solves the problem~: 
\begin{equation}\label{eq.weak.eps}
\int_{\om} 
\cL_\e(\zeps,\rhoe) 
\varphi(\bfx) \, d \bfx  
+ \int_{\om} 
\nabla \zeps (\bfx,t)\cdot \nabla \varphi (\bfx) d \bfx  = 0 
\end{equation}
for almost all $t\geq0$ and for every $\varphi \in H^1_0(\Omega)$.
\end{defi}

We need a preliminary lemma in order to show that 
our data is well prepared for the existence result.
\begin{lem}\label{lem.ci}
 Under hypotheses \ref{hypo.data.deux} and \ref{hypo.data.trois}, 
 $$
I(\bfx):= \int_{\rr} | \zp(\bfx,-\e a) | \rhoi(\bfx,a) da \in L^2(\Omega).
 $$
\end{lem}
\begin{proof}
 Using Jensen's inequality, one has
 $$
\begin{aligned}
 \int_{\Omega}  & \left\{ \int_{\rr} |\zp(\bfx,-\e a)|  \rhoi(\bfx,a) da \right\}^2 d\bfx \leq \int_{\Omega} \mu_{0,I}  \int_{\rr} |\zp(\bfx,-\e a)|^2  \rhoi(\bfx,a) da  d\bfx \\
 & \leq 2 \left( \int_{\Omega}  \int_{\rr} |\zp(\bfx,0)-\zp(\bfx,-\e a)|^2  \rhoi(\bfx,a) da  d\bfx + \int_{\Omega}   \int_{\rr} |\zp(\bfx,0)|^2  \rhoi(\bfx,a) da  d\bfx \right) \\
 & \leq 2 \left\{ \e^2 \left( \int_{\Omega} C_{\zp}^2  (\bfx) d \bfx \right) \; \left(\sup_{\bfx \in \Omega} \int_{\rr} \rhoi(\bfx,a) a^2 da \right)+ \nrm{\zp(\cdot,0)}{L^2(\Omega)}^2 \right\} \leq C
\end{aligned}
$$
which ends the proof.
\end{proof}
\begin{thm}
Under hypotheses \ref{hypo.data}, \ref{hypo.data.deux} and \ref{hypo.data.trois}, 
 there exists a unique weak solution $\zeps \in X_T$, $T$ being possibly infinite.
\end{thm}

\begin{proof}
 We define the map $\Phi$ that, given $w \in X_T$, provides $z$ being the weak solution of  the problem
\begin{equation}\label{eq.Phi}
\begin{aligned}
 ( \muze(\bfx,t)-\e \Delta_\bfx ) z(\bfx,t) = \int_0^{t/\e} & w(\bfx,t-\e a) \rhoe(\bfx,a,t) da\\
 & +\int_{t/\e}^\infty \zp(\bfx,t-\e a) \rhoe(\bfx,a,t) da, 
\end{aligned}
\end{equation}
for almost every $t\in(0,T)$.
We aim at showing that the map admits a unique fixed point using the Banach fixed point theorem. 
\begin{enumerate}
 \item $\Phi$ is endomorphic on $X_T$~:
 by Fubini one has that
 $$
\begin{aligned}
&   \int_{\Omega}   \left( \int_0^{t /\e}   w(\bfx,t-\e a) \rhoe(\bfx,a,t) da \right) v(\bfx) d \bfx 
\\ & 
\qquad = 
\int_0^{t /\e}\int_{\Omega}  w(\bfx,t-\e a) \rhoe(\bfx,a,t)  v(\bfx) d \bfx \, da \\
 & \qquad  \leq M \int_0^{t/\e} \nrm{w(\cdot,t-\e a)}{L^2(\Omega)} \nrm{v}{L^2(\Omega)} da
 \leq \frac{M t}{\e} \nrm{w}{X_T}  \nrm{v}{L^2(\Omega)} \; ,
\end{aligned}
 $$
 which proves, taking the supremum over all $v \in L^2(\Omega)$ s.t. $\nrm{v}{H^1(\om)} \leq 1$, that $\int_0^{t/ \e} $ $w(x,t-\e a)$ $ \rhoe(\bfx,$ $a,t) da$ is indeed an $L^2(\Omega)$-function.
 
 Setting $J(\bfx,t):=\int_{t/\e}^\infty \zp(\bfx,t-\e a) \rhoe(\bfx,a) da$, one has the estimate~:
 $$
\begin{aligned}
|J(\bfx,t)| \leq   \int_{t/\e}^\infty | \zp(\bfx,t-\e a) | \rhoe(\bfx,a,t) da  \leq & \int_{t / \e}^\infty | \zp(\bfx,t-\e a) | \rhoi(\bfx,a-t/ \e)  da \\
 & = \int_{\rr} |\zp(\bfx,-\e  a)| \rhoi (\bfx,a) da = I(\bfx)
\end{aligned}
 $$
 By Lemma \ref{lem.ci}, the latter term is bounded in $L^2(\om)$ by a constant $C_J$.
The right hand side in \eqref{eq.Phi} is thus an $L^2(\Omega)$ function for every time $t>0$. By the Lax-Milgram theorem, there exists
a unique solution $\zeps(\cdot,t)\in H^1_0(\Omega)$ of the problem \eqref{eq.Phi} for every fixed time $t>0$.
Moreover one has~:
$$
\min(\e,\mumin) \nrm{z(\cdot,t)}{H^1(\Omega)} \leq \frac{M t}{\e} \nrm{w}{X_T} + \nrm{I}{L^2(\Omega)} \leq \frac{M t}{\e} \nrm{w}{X_T} + C_J \; ,
$$
and taking the supremum over all times in $(0,T)$, gives~:
$$
\nrm{z}{X_T} \leq \frac{ M T}{\e \min(\e,\mumin)} \nrm{w}{X_T} + C'.
$$
This shows that $\Phi$ is an endomorphism.
\newcommand{\hw}{\hat{w}}
\item Contraction~: setting $\hz:=z_2-z_1$ (resp. $\hw:=w_2-w_1$) where $z_i = \Phi(w_i)$ for $i\in\{1,2\}$ and applying the same arguments
as above one has :
$$
\nrm{\hz}{X_T} \leq \frac{  M T}{\e \min(\e,\mumin)} \nrm{\hw}{X_T}
$$
which proves that $\Phi$ contracts as soon as $T< \e \min(\e,\mumin) / M$.
These two steps provide local existence of a fixed point $\zeps \in X_T$. 
\item Continuation~: as the time interval for which $\Phi$ is a contraction does not depend on the initial condition, we can extend the solution by continuation. This shows the global existence for any positive time $T$, possibly infinite, for $\e>0$ fixed.
\end{enumerate}
\end{proof}

\begin{coro}\label{coro.elliptic.reg}
 Under the previous hypotheses,  $\zeps \in L^\infty((0,T);H^2(\Omega))$, the bound depending on $\e^{-1}$.
\end{coro}

\begin{proof}
 The solution of the fixed point solves~:
 $$
- \Delta \zeps = \frac{1}{\e} \left\{ -\muze \zeps + \int_0^{t/\e} \zeps(\bfx,t-\e a) \rhoe(\bfx,a,t) da + J(\bfx,t) \right\},
$$
The right hand side is in $L^2(\Omega)$ for almost any time by the same arguments as above.
Because the domain $\Omega$ is smooth enough, elliptic regularity holds and 
 the claim follows (cf for instance Theorem 2.4.2.5 p.124 \cite{Grisvard}).
\end{proof}

For the rest of the article, we need to define $\dt \zeps$ and investigate to which function space it belongs.

\begin{thm}\label{thm.dtz.non.unif}
 Under the previous hypotheses, $\dt \zeps \in L^\infty((0,T);H^2(\Omega)\cap H^1_0(\om))$.
\end{thm}

\begin{proof}
 As we do not know to which space the time derivative belongs, we estimate first  a finite difference in time. Namely 
 we set 
 $$
\tet z (\bfx,t) := \frac{ z(\bfx,t+\tau)-z(\bfx,t)}{\tau} 
$$
and compute the problem it  solves~: for all $v\in \bhoz$ 
\newcommand{\mfZ}{{\mathfrak Z}}
\begin{equation}\label{eq.f.v.dt.zeps}
\begin{aligned}
 (\muze \tet \zeps,v) & + \e (\nabla \tet \zeps \nabla v) \\
 &= - \left( \left( \tet \muze \right) \zeps(\bfx,t) + \tet \int_{\rr}\zeps(\bfx,t-\e a) \rhoe(\bfx,a,t) da , v\right).
\end{aligned} 
\end{equation}
\newcommand{\mfR}{p_\e}
\noindent The product $\mfR(\bfx,a,t) := \zeps(\bfx,t-\e a) \rhoe (\bfx,a,t)$ solves the following system~: 
\begin{equation}\label{eq.R}
\left\{
 \begin{aligned}
& (\e \dt + \da + \zteps) \mfR = 0, & (\bfx,a,t) \in \Omega \times (\rr)^2 \\
& \mfR(\bfx,0,t) = \beps(\bfx,t)(1-\muze(\bfx,t)) \zeps(\bfx,t) ,& (\bfx,a,t) \in \Omega \times \{ 0\} \times \rr \\
& \mfR(\bfx,a,0) = \rhoi(\bfx,a) \zp(\bfx,-\e a) ,& (\bfx,a,t) \in \Omega  \times \rr \times \{ 0\}
\end{aligned}
\right.
\end{equation}
which is to be understood in the sense of characteristics.
One has easily in the sense of distributions, 
$$
\e \ddt{} \int_{\rr} \mfR(\bfx,a,t) da + \int_{\rr}  \zteps \mfR (\bfx,a,t) da = \beps(\bfx,t)(1-\muze(\bfx,t))\zeps(\bfx,t) .
$$
We focus on the $L^2(\Omega)$-bound of $\e \ddt{} \int_{\rr} \mfR \,  da$. Indeed~:
$$
\nrm{\int_{\rr}  \zteps \mfR (\cdot ,a,t) da}{L^2(\Omega)} \leq \frac{ M \ztmax T}{\e} \nrm{\zeps}{X_T} + \ztmax C_I \; ,
$$
whereas 
$$
\nrm{\beps (1-\muze) \zeps}{L^\infty((0,T);L^2(\Omega))} \leq \bmax \nrm{\zeps}{X_T} \; .
$$
Using Jensen's inequality and the estimate on the time derivative obtained above, one has~:
$$
\begin{aligned}
 &  \nrm{ \tet \int_{\rr} \mfR (\cdot,a,t)da }{L^2(\Omega)}^2  =   \nrm{ \frac{1}{\tau} \int_t^{t+\tau} \ddt{} \int_{\rr} \mfR(\cdot,a,s)da ds }{L^2(\Omega)}^2 \\
 &  \leq \nrm{\ddt \int_{\rr} \mfR(\cdot,a,\cdot)da}{L^\infty((t,t+\tau);L^2(\Omega))}^2 \leq \nrm{\ddt \int_{\rr} \mfR(\cdot,a,\cdot)\, da}{L^\infty((0,T);L^2(\Omega))}^2
\end{aligned}
$$
for every $t>0$.
The time derivative of $\muze$ can be estimated as follows~:
$$
\e \dt \muze = \beps(1-\muze) - \int_{\rr} \zteps \rhoe da \leq \bmax+ \ztmax < \infty
$$
and thus $\dt \muze \in L^\infty(\Omega \times \rr)$. One has then as above~:
$\left(\tet \muze\right)\zeps \in$ $L^\infty(\rr;$ $L^2(\Omega))$, which gives by Lax-Milgram applied to \eqref{eq.f.v.dt.zeps}~:
$$
\min(\e,\mumin) \nrm{\tet \zeps(\cdot,t)}{H^1(\Omega))}<C
$$
for every fixed $t\in[0,T]$. 
Moreover, by standard elliptic regularity and since the right hand side 
in \eqref{eq.f.v.dt.zeps} is an $L^2(\om)$ function, 
$\nrm{\tet \zeps(\cdot,t)}{H^2(\Omega))} < \infty$.
Thus, modulo the extraction of a subsequence $(\tau_k)_{k \in \N}$, there exists $L^\infty((0,T);H^2(\Omega))$ weak-$*$ limit which is 
a weak time derivative of $\zeps$ (see for instance Theorem 3 Section 5.8.2. \cite{Evans.Book}), and the derivative
satisfies the same $L^\infty((0,T);H^2(\Omega))$ bound.
\end{proof}

\begin{rem}
   Estimates above are not uniform with respect to $\e$. 
 These computations are  performed only in order to give a meaning to the time derivative of $\zeps$, 
 and show that locally with respect to $\e$ it is an $L^\infty_t H^2_{\bfx}$ function.
\end{rem}

\section{Energy estimates}\label{sec.nrj}

\subsection{The energy $\cet$ decreases with time}
\begin{thm}\label{thm.nrj}
Under hypotheses \ref{hypo.data}, \ref{hypo.data.deux} and \ref{hypo.data.trois}, 
 for all times $t \geq 0$, the energy $\cet$ is a decreasing function, {\em i.e}~:
 $$
 \cet(\zeps(\cdot,t)) \leq \cE_0(\zp(\cdot,0)) \; .
 $$
 Moreover, one has as well that
 $$
 \int_0^T \int_{\Omega\times\rr} \zteps(\bfx,a,t)  \rhoe(\bfx,a,t) \left( \frac{ \zeps(\bfx,t)-\zeps(\bfx,t-\e a)}{\e} \right)^2 da \, d\bfx \, dt < \cE_0(\zp(\cdot,0)).
 $$
 
\end{thm}

\begin{proof}
We use again the same procedure in order to pass from the position to the elongation as in \cite{MiOel.2,MiOel.3}, writing~:
\newcommand{\unk}{u_\e}
\newcommand{\cnk}{{\zeps}}
\begin{equation}
  \label{eq.def.veps}
\unk(\bfx,a,t):=
\begin{cases}
 \frac{\zeps(\bfx,t)-\zeps(\bfx,t-\varepsilon a)}{\varepsilon} & \text{ if } t \geq \e a \; , \\
  \frac{\zeps(\bfx,t)-\zp(\bfx,t-\varepsilon a)}{\varepsilon}& \text{ otherwise} \; .
\end{cases}  
\end{equation}
Indeed, so defined $\veps$ solves 
\begin{equation}\label{eq.u.z}
(\e \dt  + \da )\unk  = \dt \cnk \; ,
\end{equation}
this equation has a meaning in the sense of characteristics, while the right hand side is 
ment as a function in $L^\infty((0,T);H^2(\Omega))$ as shown in the previous section.

Considering the equation satisfied by $\rhoe \veps^2$ and integrating in age gives~:
$$
\frac{\e}{2} \ddt{} \int_{\rr} \rhoe \unk^2 da + \int_{\rr} \zteps \rhoe \unk^2  da = \left( \int_{\rr} \rhoe \veps da \right) \dt \zeps
= \Delta \zeps \dt \zeps \; ,
$$
which integrated in space gives~:
\begin{equation}\label{eq.fund}
\begin{aligned}
  \frac{\e}{2}  \ddt{} \int_{\rr\times\Omega} \rhoe \unk^2 da d \bfx & + \int_{\rr\times \Omega} \zteps \rhoe \unk^2  da d\bfx \\
 & =  -\int_{\Omega} \nabla \zeps \nabla \dt \zeps d\bfx = - \nud \ddt{} \int_\Omega | \nabla \zeps |^2  d\bfx \; .
\end{aligned}
\end{equation}
The latter integration by parts is justified as follows. Set $w_\e := \nabla \zeps$, thanks to Corollary \ref{coro.elliptic.reg} and Theorem \ref{thm.dtz.non.unif}, one has that $w_\e \in W^{1,\infty}([0,T];L^2(\om))\subset C([0,T];L^2(\om))$.
The latter space is separable~: there exists a $C^\infty([0,T]\times\om)$ function 
s.t. $w_\e^\delta \to w_\e$ in $C([0,T];L^2(\om))$ strong, and $\dt w_\e^\delta \rightharpoonup \dt w_\e$ in 
$L^\infty((0,T);L^2(\om))$  weak-* ($w^\delta_\e$ can be obtained by the standard mollification). 
In this scenario, one is testing against a $C^1$ function in time, the integration by parts on the regularized functions. Passing to the limit with respect to $\delta$, leads to~:
$$
\int_0^T \varphi(t) \int_\om 2 w_\e \dt w_\e  d\bfx \, dt =\left[ \int_\om | w_\e(\bfx,t) |^2 d\bfx  \varphi(t) \right]_{t=0}^{t=T}  - \int_0^T \int_\om | w_\e |^2 d\bfx \,\dt \varphi \,dt 
$$
for any $\varphi \in C^1([0,T])$. As $\int_\om | w_\e(\bfx,t) |^2 d\bfx$ is an absolutely continuous function of $t$, the integration by part  holds, and thus 
$$
2 \int_\om w_\e \dt w_\e d \bfx= \ddt{} \int_\om |w_\e|^2 d \bfx,\quad \text{ for }a.e. \quad  t\in(0,T).
$$
Finally \eqref{eq.fund}  gives~:
$$
\ddt{} \cet(\zeps(\cdot,t)) \leq 0 \; ,
$$
since $\int_{\rr\times \Omega} \zteps \rhoe \unk^2  da d\bfx$ is positive.
But as $\zeps(\bfx,0)$ solves \eqref{eq.z.eps} at time $t=0$,  by Lemma \ref{lem.equiv},  $\zeps(\bfx,0)$ minimizes  
the energy at time $t=0$. This proves that
$$
\cet(\zeps(\cdot,t)) \leq \cE_0(\zeps(\cdot,0)) \leq \cE_0(\zp(\cdot,0)),
$$
giving the first claim provided the last term is bounded. 
But, by similar arguments as in Lemma \ref{lem.ci}, one has  that
$$
\cE_0(\zp(\cdot,0)) \leq \e \left( \int_{\Omega} C_{\zp}^2  (\bfx) d \bfx \right) \; \left(\sup_{\bfx \in \Omega} \int_{\rr} \rhoi(\bfx,a) a^2 da \right) + \int_{\Omega} | \nabla \zp(\bfx,0) |^2 d \bfx < \infty
$$
the last term being bounded since $\zp(\bfx,0) \in H^1_0(\Omega)$.
Integrating \eqref{eq.fund} in time gives~:
$$
\int_0^T \int_{\rr\times \Omega} \zteps \rhoe \unk^2  da d\bfx dt \leq \cE_0(\zeps(\cdot,0))-\cE(\zeps(\cdot,t)) \leq \cE(\zeps(\cdot,0)) \leq \cE(\zp(\cdot,0)) \; ,
$$
which ends the proof.
\end{proof}

\begin{coro}
 Under the same hypotheses, $\zeps \in X_T$  uniformly with respect to $\e$.
\end{coro}

\begin{proof}
  The bound on  the gradient is completed by the norm of $\zeps(\cdot,t)$ in $L^2(\Omega)$ by  the Poincar\'e inequality.
\end{proof}

\begin{thm}\label{thm.dt.zeps.l2}
 Under the same hypotheses as above,  $\dt \zeps$ in $L^2(Q_T)$ and the bound is uniform in $\e$.
\end{thm}

\begin{proof}
Multiplying $\veps$ by $\rhoe$ it solves in the sense of characteristics~:
$$
(\e \dt + \da + \zteps) \rhoe \veps = \rhoe \dt \zeps.
$$
Integrating  with respect to the age variable, and because $\veps(\bfx,0,t)=0$, one has 
$$
\e \dt \int_{\rr} \rhoe \veps da + \int_{\rr} \zteps \rhoe \veps da = \muze \dt \zeps \; .
$$
We recall that $\zeps$ solves~: 
$$
\int_\om \int_{\rr} \rhoe  (\bfx,a,t)\veps (\bfx,a,t)  da \; v(\bfx) d \bfx + \left( \nabla \zeps, \nabla v\right)=0, \quad \forall v\in H^1_0(\om)
$$
for almost every fixed $t\in(0,T)$.
Due to Theorem  \ref{thm.dtz.non.unif}, $(\nabla \zeps,\nabla v)$ is a differentiable function in time for any $v\in H_0^1(\om)$ and thus
$$
\e \dt \int_\om \int_{\rr} \rhoe  (\bfx,a,t)\veps (\bfx,a,t)  da \; v(\bfx) d \bfx + \e \left( \nabla \dt \zeps,\nabla v\right)=0, \quad \forall v\in H^1_0(\om) \; .
$$
This shows that $\dt \zeps$ solves indeed
$$
\int_{\om} \muze \dt \zeps(\bfx,t) v(\bfx) d\bfx + \e \int_{\om} \nabla \dt \zeps \cdot \nabla v d\bfx = \int_\om \left( \int_{\rr} \rhoe \zteps \veps da \right) v(\bfx) d \bfx
$$ 
  for every fixed $t>0$ and any $v \in H^1_0(\om)$.
On the other hand, using Jensen's inequality one has
$$
\left( \int_\rr \zteps \rhoe | \veps | da  \right)^2 \leq \int_{\rr} \zteps \rhoe da \int_{\rr} \zteps \rhoe \veps^2 da \; ,
$$
which integrated in space and time gives 
$$
\begin{aligned}
 \nrm{ \int_{\rr} \rhoe \zteps \veps da }{L^2(Q_T)}^2&  \leq \left( \sup_{(\bfx,t)\in Q_T} \int_{\rr} \zteps \rhoe da \right)\int_0^T \int_\om \int_{\rr} \zteps \rhoe \veps^2 da \, d\bfx \, dt \\
 & \leq \ztmax \int_0^T \int_\om \int_{\rr} \zteps \rhoe \veps^2 da \, d\bfx \, dt \; .
\end{aligned} 
$$
By Lax-Milgram, one has the estimates~:
$$
\nrm{\dt \zeps(\cdot,t)}{L^2(\om)} \leq \frac{1}{\mumin} \nrm{\int_{\rr} \rhoe(\cdot,a,t) \zteps(\cdot,a,t) \veps(\cdot,a,t) da }{L^2(\om)}
$$
for almost every $t \in (0,T)$,
which gives after integration in time  that $\dt \zeps \in L^2 (Q_T)$ uniformly with respect to $\e$. 
\end{proof}

\begin{coro}\label{coro.aubin}
 Under the previous hypotheses, there exists a subsequence $(z_{\e_k})_{k \in \N}$ converging strongly in $C( [0,T];L^2(\Omega))$.
\end{coro}

\begin{proof}
 The imbedding $H^1_0(\Omega)$ is compact in $L^2(\Omega)$
which gives by the Lions-Aubin-Simon theorem that there exists a subsequence $(z_{\e_k})_{k \in \N}$ converging strongly in $C( [0,T];L^2(\Omega))$
(cf. Theorem II.5.16 p.102 \cite{FaBo.Book}).
\end{proof}

\subsection{A stability result in the elongation variable}
The problem solved by $\veps$ reads formally~:
\newcommand{\uei}{u_{\e,I}}
\begin{equation}\label{eq.v.eps}
\hspace{-0.54cm}
\left\{
\begin{array}{l l l}
&  (\e \dt + \da) \veps = \dt \zeps(\bfx,t), & (\bfx,a,t) \in \Omega\times\rr\times(0,T) \\
& (\muze -\e \Delta ) \dt\zeps =  \int_{\rr} (\zteps \rhoe \veps)(\bfx,a,t)da, & (\bfx,t) \in \om\times(0,T),\\
& \dt \zeps(\bfx,t) = 0 & (\bfx,t) \in \om\times(0,T), \\ 
 &\veps(\bfx,0,t) = 0 & (\bfx,a,t)  \in \Omega \times \{ a=0\} \times (0,T) \\
 &\veps(\bfx,a,t) =0 &  (\bfx,a,t) \in \dom\times \rr \times (0,T) \\
 &\veps(\bfx,a,0) = \uei (\bfx,a) &  (\bfx,a,t)\in \Omega \times\rr\times   \{t= 0\}
\end{array}
\right.
\end{equation}
where $\uei (\bfx,a):=\frac{ \zeps(\bfx,0)-\zp(\bfx,- \e a) }{\e}$ and $\zeps(\bfx,0)$ solves
\begin{equation}\label{eq.zeps.t.0}
 ( \mu_{0,I}(\bfx)-\e \Delta_\bfx ) z(\bfx,0) =  \int_{0}^\infty \zp(\bfx,-\e a) \rhoi(\bfx,a) da.
\end{equation}
The elliptic problem solved by $\dt\zeps$ in \eqref{eq.v.eps} is to be understood in the variational sense.
This system has to be compared with (2.1) p.5 \cite{MiOel.2}, 
here the inverse of the operator $\left(\muze I - \e \Delta\right)$ appears 
as a space contribution. In what follows we show how 
to deal with and extend stability estimates (2.6) p.6 \cite{MiOel.2} in this setting.
\begin{thm}\label{thm.stab.rho.veps}
 Under hypotheses \ref{hypo.data} and \ref{hypo.data.deux}, and if 
 $
 \int_{\rr} \rhoi | \vepsi | da \, d\bfx < \infty,
 $
  one has~:
 $$
 \int_{\om\times\rr} (\rhoe | \veps | )(\bfx,a,t) da d\bfx \leq \int_{\om\times\rr} \rhoi (\bfx,a) | \vepsi(\bfx,a) | da\,  d \bfx \; .
 $$
 Moreover, if $\vepsi$ satisfies
 $$
 \sup_{a\in \rr} \frac{\int_{\Omega} | \vepsi (\bfx,a) | d\bfx }{(1+a)} < \infty,
 $$ 
 then
 $$
 \int_{\Omega} | \veps(\bfx,a,t) | d\bfx \in Y_T:=L^\infty\left(\rr\times(0,T) , \frac{ 1}{1+a} \right)
 $$
 and the bound is uniform with respect to $\e$. 
\end{thm}
\begin{proof}
 A simple use of Theorem \ref{thm.nrj}, shows that 
 $$
 \int_{\Omega\times\rr} \zteps \rhoe \veps^2 da d\bfx \leq  \ztmax  \int_{\Omega\times\rr}  \rhoe \veps^2 da d\bfx \leq \frac{\ztmax}{\e} \cE(\zeps(\cdot,t) \leq \frac{\ztmax}{\e} \cE(\zp(\cdot,0)),
 $$
 which, using again Jensen's inequality, implies that
 $$
\int_{\Omega} \left( \int_{\rr} \zteps \rhoe |\veps | da \right)^2 d\bfx \leq  \frac{\ztmax^2}{\e} \cE(\zp(\cdot,0)).
$$
This bound ensures that for fixed $\e$, $f(\bfx,t):=\int_{\rr} \zteps \rhoe \veps  da$ belongs to $L^\infty((0,T);$ $L^2(\Omega))$.
We consider the problem~: for a given $f(\bfx,t)$ find $g(\bfx,t)$ solving
$$
\left\{
\begin{aligned}
& \muze g- \e \Delta g = f,  &  \text{ in } \Omega\;,\\
& g=0,& \text{ on } \dom\;.
\end{aligned}
\right.
$$
For almost every $t\in (0,T)$, one solves this elliptic problem. Thus
there exists a unique $g \in L^\infty((0,T);H^2(\Omega)\cap H^1_0(\Omega))$ by Lax-Milgram and standard elliptic regularity. These 
considerations allow to fulfill hypotheses of the main theorem in \cite{BrePon}, namely for a.e. $t \in (0,T)$,
$g(\cdot,t) \in L^1(\Omega)$, $\Delta g (\cdot,t)  \in L^1(\Omega)$ and $\ddn g (\cdot,t) \in L^1(\dom)$ which ensures that $g(\cdot,t) \in {\mathbb X}$
where 
$$
{\mathbb X}:=\left\{ u \in W^{1,1}(\Omega) \; \st \; \left| \int \nabla u\cdot \nabla \psi d\bfx \right| < C \nrm{\psi}{L^\infty(\Omega)} \; \forall \psi \in C^1(\ov{\Omega})\right\}
$$
and thus a {\em Green's inequality} holds (cf. Theorem 1.3, \cite{BrePon})~:
$$
\int_{\Omega} \nabla g^+  \cdot \nabla \psi d\bfx \leq \int_{\dom} H \psi  - \int_{\Omega} G \psi,\quad \forall \psi \in C^1(\ov{\om}) \;, \quad \psi \geq 0,
$$
where $g^+$ denotes the positive part of $g$ and $G \in L^1(\Omega)$ and $H\in L^1(\dom)$ are given by~:
$$
G:= 
\begin{cases}
 \Delta g&  \text{ on } \{ g>0 \} \\
 0 & \text{ on } \{ g\leq 0 \} \\
\end{cases}, \quad 
H := 
\begin{cases}
 \ddn g & \text{ on } \{ g >0\} \;, \\
 0 & \text{ on } \{ g<0 \} \; , \\
 \min \left( \ddn g,0\right) & \text{ on } \{ g =0 \} \; .
\end{cases}
 $$
Applying the latter result to $|g|:=g^+-g^-$, since $g$ vanishes on the boundary, one obtains that 
$$
\int_{\Omega} \; \Delta g \; \sgn g \; d \bfx \leq 0.
$$
Returning to \eqref{eq.v.eps}, one has 
$$ 
\e \dt \veps + \da \veps = g \; ,
$$
where we set $g:=\dt \zeps$.
In the sense of characteristics, one establishes, after integration with respect to age~:
$$
\e \dt \int_{\rr} \rhoe | \veps | da + \int_{\rr} \zteps \rhoe | \veps | da \leq \muze | g | \; .
$$
Integrating  in space, one obtains 
$$
\e \ddt{} \int_{\Omega\times\rr} \rhoe | \veps | da d\bfx + \int_{\Omega\times\rr} \zteps \rhoe | \veps | da d \bfx \leq \int_{\Omega} \muze | g | d \bfx \; .
$$
But then
$$
\int_{\Omega} \muze | g | d \bfx = \int_{\Omega} \left( \int_{\rr} \zteps \rhoe \veps da \right)\sgn g \, d\bfx + \e \int_{\Omega}   \Delta g\; \sgn g\, d\bfx \leq \int_{\Omega \times \rr} \zteps \rhoe \left| \veps \right| da \, d\bfx.
$$
This leads to 
$$
\e \ddt{} \int_{\Omega\times\rr} \rhoe | \veps | da \, d\bfx \leq 0 \; ,
$$
which, after integration in time, proves the first result. 
Then, one has that $q(a,t):=\int_{\Omega} | \veps | d \bfx$ solves 
$$
\begin{aligned}
 \e \dt q + \da q & \leq \frac{1}{\mumin} \int_{\Omega} \muze |v| d\bfx\leq \frac{1}{\mumin}  \int_{\Omega \times \rr} \zteps \rhoe \left| \veps \right| da d\bfx \\
 & \leq \frac{\ztmax}{\mumin} \int_{\rr} \rhoi|\vepsi| da d\bfx < C \; .
\end{aligned}
$$
Applying then the same results as in Theorem 6.1 \cite{MiOel.2}, one concludes that $q \in Y_T$.
\end{proof}

It remains to show that 
the assumptions of theorem \ref{thm.stab.rho.veps}
are fulfilled. This is the scope of next two lemmas.
\begin{lem}\label{lem.ci.veps.stab}
 Under assumptions \ref{hypo.data.trois} it holds that~:
 $$
 J:= \int_{\Omega\times \rr} \rhoi | \vepsi | da d \bfx < C \; ,
 $$
 where the generic constant C is finite and  independent on $\e$.
\end{lem}

\begin{proof}
 A triangle inequality gives~:
 $$
 J \leq \int_{\Omega} \frac{\left| \zeps(\bfx,0)-\zp(\bfx,0)\right|}{\e} \mu_{0,I} (\bfx)\; d\bfx + \int_{\Omega\times\rr} \frac{\left| \zp(\bfx,0)-\zp(\bfx,-\e a)\right|}{\e} \rhoi (\bfx,a)da \,d\bfx 
 \; .
 $$
By similar arguments as above, one considers the problem solved by $\hz(\bfx,0):=\zeps(\bfx,0)-\zp(\bfx,0)$~:
 $$
 \mu \hz(\bfx,0) -\e \Delta \hz(\bfx,0) = -\int_{\rr} \left(\zp(\bfx,0)-\zp(\bfx,-\e a) \right) \rhoi (\bfx,a) da + \e \Delta \zp(\bfx,0) .
 $$
 Since $\zp(\cdot,0)$ is in $\bhoz$, the right hand side is in $H^{-1}(\om)$, thus 
 by Lax-Milgram, $\hz(\bfx,0) \in \bhoz\subset W^{1,1}_0(\om)$.
 Moreover since the right hand side is in $L^1_\bfx$ as well, 
one fulfills  the 
 hypotheses of Proposition 4.2 \cite{BrePon}  which shows that $\hz(\bfx,0) \in {\mathbb X}$ and 
  $\ddn \hz(\cdot,0) \in L^1(\dom)$. Again  Theorem 1.3 \cite{BrePon} applies and one obtains that
 $$
\nrm{\mu_{0,I}  \hz(\cdot,0) }{L^1_\bfx} \leq \nrm{\int_{\rr} \left(\zp(\bfx,0)-\zp(\bfx,-\e a) \right) \rhoi (\bfx,a) da }{L^1_\bfx} + \e \nrm{\Delta\zp(\cdot,0)}{L^1_\bfx} \; ,
$$
which together with the Lipschitz-like  assumption \ref{hypo.data.trois} $(ii)$   ends the proof.
\end{proof}

\begin{lem}
 Under assumptions \ref{hypo.data.trois}, one has also that the second requirement on $\vepsi$ holds~:
 $$
 \sup_{a\in \rr} \frac{\int_{\Omega} | \vepsi (\bfx,a) | d\bfx }{(1+a)} < C,
 $$ 
 where the generic constant is independent on $\e$.
\end{lem}
\begin{proof}
 The same triangle inequality holds but we do not integrate in age~:
$$
\begin{aligned}
 \int_{\Omega}&  | \vepsi| d\bfx  \leq  \int_{\Omega} \frac{\left| \zeps(\bfx,0)-\zp(\bfx,0)\right|}{\e} d\bfx + \int_{\Omega} \frac{\left| \zp(\bfx,0)-\zp(\bfx,-\e a)\right|}{\e}  d\bfx \\
 &  \leq  \int_{\Omega} \mu_{0,I} \frac{\left| \zeps(\bfx,0)-\zp(\bfx,0)\right|}{\e \mumin } d\bfx + a \int_{\Omega} C_{\zp}(\bfx)d \bfx \leq C + a \sqrt{| \Omega|} \nrm{C_{\zp}}{L^2(\Omega)} \; .
\end{aligned}
 $$
 Dividing by $(1+a)$ and taking the supremum on $\rr$ ends the proof.
\end{proof}
\begin{lem}
 Under hypotheses above, one has also that $\nrm{\dt \zeps}{L^\infty_t L^1_\bfx}<\infty$ uniformly in $\e$.
\end{lem}
\begin{proof}
 By Theorem \ref{thm.stab.rho.veps} and the hypotheses on $\zteps$ one has that
 $\int_{\om\times \rr} \rhoe \zteps \veps \, da \, d\bfx \in C([0,T])$. Since $\Delta\dt \zeps$
 belongs for almost every $t\in(0,T)$ to $L^2(\om)$ by Theorem \ref{thm.dtz.non.unif}, 
 we satisfy the hypotheses of Theorem 1.3, \cite{BrePon} and we conclude that
 $$
\int_\om \muze(\bfx,t) | \dt \zeps(\bfx,t)  | d\bfx \leq  \int_\om \left| \int_\rr (\rhoe \zteps \veps)(\bfx,a,t) da \right| d \bfx \; .
$$
Finally, taking the ess-sup in time,  one concludes the proof.
\end{proof}

\section{Convergence when $\e$ goes to zero}
\label{sec.conv}

Since the system \eqref{eq.rho.eps}-\eqref{eq.z.eps} is weakly coupled,
and  the space variable $\bfx$ is a mute parameter for the density 
of linkages $\rhoe$, the convergence results from
the previous articles are adapted and attention is
paid only on the order of functional spaces with respect to $\bfx$, $a$
and $t$ in section \ref{sec.conv.rho}. Then in section \ref{sec.conv.z}, we present the
main result of the first part of the paper.

\subsection[Convergence of linkages]{Convergence of $\rhoe$}\label{sec.conv.rho}

Concerning the  convergence of  $\rhoe$, we recall the Lyapunov functional, \cite{MiOel.1}~:
\begin{equation}
\label{defliapunovfunc} 
\mathcal{H}[u(\bfx,\cdot)]:=  \left| \int_0^\infty  u(\bfx,a) \; da \right|  +\int_0^\infty | u(\bfx,a) | \; da\; ,
\end{equation}
for every $a$-measurable function $u(\bfx,\cdot)$.
Consider the difference $\hrhoe := \rhoe-\rho_0$. A formal computation using \eqref{eq.rho.eps} and \eqref{eq.rho.zero} implies that it satisfies
\begin{equation}\label{equ_rhohat_residual} 
\left\{ 
\begin{aligned} 
&\varepsilon \dt \hrhoe + \da \hrhoe + \zeta_\varepsilon (\bfx,a,t) \hrhoe = \mathcal{R}_{\varepsilon,\bfx} \;,& a > 0 \, , \; t > 0\; ,\\
& \hrhoe(\bfx,a=0,t)=-\beta_{\varepsilon}(\bfx,t) \int_0^\infty\hrhoe(\bfx,\tilde a, t) \; d\tilde a + \mathcal{M}_{\varepsilon,\bfx} \; ,&\quad t>0\;, \\
& \hrhoe(\bfx,a,t=0)= \rho_{\varepsilon,I}(\bfx,a)-\rho_{0}(\bfx,a,0) \; ,& \quad a\geq0 \;,
\end{aligned}  
\right. 
\end{equation}
with $\mathcal{R}_{\varepsilon}(\xat):=-\varepsilon  \partial_t \rho_0(\bfx,a,t)-\rho_0(\bfx,a,t)(\zteps(\bfx,a,t)-\zeta_0(\bfx,a,t))$ and
$ \mathcal{M}_{\varepsilon}(\bfx,t):=(\beps(\bfx,t)-\beta_0(\bfx,t))$ $\left(1-\int_0^\infty \right.$ $\left.\rho_0(\bfx,a,t) \, da \right)$.

\begin{lem}
\label{lem_convergence_rho}
According to assumptions~\ref{hypo.data}, one has~:
\begin{displaymath}
\mathcal{H}[\hrhoe(\bfx,\cdot,t)]   \leq 
\mathcal{H}[\rho_{\varepsilon,I}(\bfx,\cdot)-\rho_{0}(\bfx,\cdot,0)]   e^{\frac {-\ztmin  t }{\varepsilon}}   +
\frac{2}{\ztmin} \left\{ 
\nrm{\mathcal{R}_{\varepsilon}}{L^\infty_{\xat}}+\nrm{\cM_\e}{L^\infty_{\bfx,t}}
\right\}
\end{displaymath}
for all $t \geq 0$ and a.e. $\bfx \in \Omega$.
\end{lem}

Using the method of characteristics one can also write pointwise estimates~:
\begin{lem}\label{lem.duh.hrho}
 One can estimate the difference $\hrhoe$ locally with respect to  $(\bfx,a,t)$~:
 $$
\left| \hrhoe (\bfx,a,t) \right| \lesssim  
\begin{cases}
 \bmax \exp\left(  - \frac{ \ztmin t}{\e} \right) \nrm{\hrhoi(\bfx,\cdot)}{L^1_a} 
   +  (1+a)^2\exp(-\ztmin a) & \text{ if } t \geq \e a \\
 | \hrhoi \left(\bfx,a-\frac{t}{\e}\right) | \exp\left(  - \frac{ \ztmin t}{\e} \right) 
   +  (1+a)^2 \exp\left( - a \ztmin \right) & \text{ otherwise} 
  \end{cases}
$$
for almost every $\bfx \in \Omega$ and $L^1_a:=L^1(\rr)$.
\end{lem}
\begin{proof}
 We use  Duhamel's formula and write~:
\begin{itemize}
 \item if $t \geq \e a$, 
$$
\begin{aligned}
| \hrhoe (\bfx,a,t)| \leq & | \hrhoe(\bfx,0,t-\e a) | \exp( -\ztmin a) \\
& + \int_0^a \exp( - \ztmin (a-s) ) \left| \cR_{\e,\bfx} (\bfx,s,t + \e ( s -a ) ) \right|ds  \\
\leq &  | \hrhoe(\bfx,0,t-\e a) | \exp( -\ztmin a) \\
& + o(1) \int_0^a \exp( - \ztmin (a-s) )(1+s) \exp( - \ztmin s ) ds  \\ 
\leq & \left\{  | \hrhoe(\bfx,0,t-\e a) | + o(1) (1+a)^2 \right\}  \exp(-\ztmin a)
\end{aligned}
$$
where we used Lemma \ref{lem.rhoz.est} in the integral part of the right hand side.
Then, thanks to Lemma \ref{lem_convergence_rho}, the first term can be estimated as
$$
\begin{aligned}
 | \hrhoe(\bfx,0,t-\e a) | & \exp( -\ztmin a) \leq \left(\bmax | \hmu(\bfx,t-\e a) | + \left| \cM_{\e,\bfx} \right| \right)\exp( -\ztmin a)\\
\end{aligned}
$$ 
which using again Lemma \ref{lem_convergence_rho} gives~:
$$
\begin{aligned}
& | \hrhoe(\bfx,0,t-\e a) |  \exp( -\ztmin a) \leq \\
 & \leq \left(\bmax \nrm{\hrhoi}{L^\infty_\bfx L^1_t} \exp\left(-\frac{\ztmin(t-\e a)}{\e} \right) 
 + o(1) + \left| \cM_{\e,\bfx} \right| \right)\exp( -\ztmin a)\\
 & \leq C_1 \exp\left(-\frac{\ztmin t}{\e} \right) + o(1) (1+a) \exp( -\ztmin a)
\end{aligned}
$$ 
where $L^\infty_\bfx L^1_t:=L^\infty(\om; L^1(\rr))$.
 \item if $t \leq \e a$, the claim follows from  Duhamel formula and Lemma \ref{lem.rhoz.est} directly.
\end{itemize}
~\end{proof}
\begin{coro}
 Under hypotheses \ref{hypo.data} and \ref{hypo.data.deux},  one has that
 $$
 \int_{\rr} \sup_{\bfx \in \Omega} |\rhoe(\bfx,a,t) - \rhoz(\bfx,a,t )| da \leq C\left( 1+\frac{  t }{\e} \right) \exp\left( -\frac{ \ztmin t}{\e} \right) + o(1) \; ,
 $$
 which means that $\sup_{\bfx} |\rhoe(\bfx,a,t) - \rhoz(\bfx,a,t )|$ converges strongly in $L^1((0,T)\times \rr)$.
\end{coro}

\begin{proof}
The proof follows by integrating in age the previous Lemma.
\end{proof}
\begin{coro}\label{coro.conv.hrho.a}
 Under the same hypotheses,  $\sup_{\bfx \in \Omega} |\rhoe(\bfx,a,t) - \rhoz(\bfx,a,t )|$ converges strongly in $L^1((0,T)\times \rr, (1+a) )$.
\end{coro}

\subsection[Convergence of posititions]{Convergence of $\zeps$}\label{sec.conv.z}

\begin{thm}\label{thm.cvg}
 Under hypotheses \ref{hypo.data}, \ref{hypo.data.deux} and \ref{hypo.data.trois}, the weak solution $\zeps$ (cf Definition \ref{def.weak.sol.z.eps}) tends to $\zz \in L^\infty([0,T];H^1_0(\om))$ with $\dt \zz \in L^2(\qt)$, the  weak solution of \eqref{eq.z.0}, {\em i.e.}
 \begin{equation}\label{eq.f.v.z.0}
\int_{\qt} \mu_{1,0} \dt \zz \varphi(\bfx,t) d\bfx dt + \int_{\qt} \nabla \zz \cdot \nabla \varphi d\bfx dt =0 \; .
\end{equation}
for every  test function $\varphi \in \dot{H}^1(\qt) := \{ u \in H^1(\qt) \st u=0 \; $ $ \text{ a.e. in }  \dom \times (0,T)\}$
\end{thm}
\begin{proof}
We test the weak formulation in Definition \ref{def.weak.sol.z.eps}  by a function $v\in H^1_0(\om)\cap L^\infty(\om)$ and we integrate in time after testing by $w \in L^\infty((0,T))$. Rewriting in terms of the elongation variable we obtain~:
\begin{equation}
\label{eq.weak.eps.time} 
\int_{\qt} \int_\rr \rhoe(\bfx,a,t) \veps(\bfx,a,t) \, da \,  v(\bfx) \,  d\bfx \,  w(t) \,  dt + \int_{\qt} \nabla \zeps(\bfx,t) \cdot \nabla v(\bfx) \, w(t) \,  d\bfx  \, dt =0 \; .
\end{equation}
We denote $\varphi(\bfx,t):=v(\bfx)w(t)$ and start with the convergence of the first term above
 $$
\begin{aligned}
&  \int_{\Omega} \int_0^T \int_{\rr} \rhoe(\bfx,a,t) \veps(\bfx,a,t)  \varphi(\bfx,t)\, da \; dt\;d\bfx \\
& =  \int_{\Omega} \int_0^T \int_{\rr} (\rhoe(\bfx,a,t) - \rhoz(\bfx,a,t) ) \veps (\bfx,a,t) \varphi(\bfx,t)\, da \; dt\;d\bfx  \\
 & \quad + \int_{\Omega} \int_0^T \int_{\rr} \rhoz (\bfx,a,t) \veps(\bfx,a,t)  \varphi(\bfx,t) \, da \; dt\;d\bfx =: I_1 + I_2 \; .\\
\end{aligned} 
 $$
Due to Corollary \ref{coro.conv.hrho.a} and Theorem \ref{thm.stab.rho.veps}, $I_1$ can be estimated~:
 $$
 \left| I_1 \right| \leq \nrm{\hrhoe (1+a)}{L^1\left((0,T)\times\rr;L^\infty(\Omega)\right)} \nrm{\frac{\veps}{1+a}}{L^\infty\left((0,T)\times\rr;L^1(\Omega)\right)} \nrm{\varphi}{L^\infty(\qt)} \sim  o_\e(1) \; .
 $$
 For the second term one has~:
 $$
\begin{aligned}
& \int_{Q_T} \int_0^\infty \rhoz(\bfx,a,t) \veps(\bfx,a,t) \varphi (\bfx,t)\; da \;d\bfx\; dt \\ 
& = \int_0^{T/\e}  \left( \int_{\e a}^T \int_{\Omega} \rhoz \frac{\veps}{a} \varphi d\bfx dt  \right) a da 
 +\int_0^T  \int_{t/ \e}^\infty \int_\om  \rhoz(\bfx,a,t)\veps(\bfx,a,t) \varphi (\bfx,t)\;d\bfx \; da \; dt\\
& = I_{2,1}+I_{2,2} \; .
\end{aligned}
$$
The first part of this expression can be rewritten as~:
$$
I_{2,1} = \int_0^{T/\e} \left( \int_{\e a}^T \int_{\Omega} \rhoz \frac{(\zeps(\bfx,t) -\zeps(\bfx,t-\e a))}{\e a} \varphi (\bfx,t)d\bfx \,dt  \right) \, a \,da \; .
$$
For almost every fixed $a \in \rr$, one has convergence of the term
$$
 \int_{\e a}^T \int_{\Omega} \rhoz \frac{(\zeps(\bfx,t) -\zeps(\bfx,t-\e a))}{\e a} \varphi (\bfx,t) d\bfx dt \to  \int_{0}^T \int_{\Omega} \rhoz (\bfx,a,t)\dt \zz (\bfx,t)\varphi (\bfx,t)d\bfx
$$
because of the weak convergence in $L^2(Q_T)$ of the sequence $\frac{(\zeps(\bfx,t) -\zeps(\bfx,t-\e a))}{\e a}$.
Moreover, thanks to the estimates on $\rhoz$, one has that
$$
\begin{aligned}
f_\e(a) & :=   a  \int_{\e a}^T \int_{\Omega} \rhoz \frac{(\zeps(\bfx,t) -\zeps(\bfx,t-\e a))}{\e a}  \varphi \, d\bfx \, dt  \\
  & \leq C \,a \,\exp\left( -  \ztmin a  \right)  \nrm{\chiu{(\e a,T)} D^{-\e a}_t \zeps}{L^2(\qt)} \nrm{\varphi}{L^2(Q_T)} \\
  & \lesssim  a \exp(-\ztmin a) \sup_{\tau\in(0,T)} \nrm{\chiu{(\tau,T)} D^{-\tau}_t \zeps}{L^2(\qt))}\lesssim  a \exp(-\ztmin a)\nrm{\dt \zeps}{L^2(\qt)} \;.
\end{aligned}
 $$
Due to Theorem \ref{thm.dt.zeps.l2} the norm $\dt\zeps$ is 
 bounded uniformly in $\e$, and thus the majorizing function is a $L^1$ function in age. 
 Applying Lebesgue's Theorem  gives the commutation of the limit and the integral in age of $f_\e$.
 
 With regard to the rest, we set
 $
 I_{2,2} 
 =: \int_0^T \heps(t) dt
 $
and infer that
$$
\begin{aligned}
 \heps(t) & \leq \int_{t/ \e}^\infty C \exp(-\ztmin a) (1+a) \sup_{a\in \rr} \frac{\int_\Omega | \veps | d\bfx }{(1+a)} \nrm{\varphi}{L^\infty(Q_T)} da \\
 & \leq C \left( 1 + \frac{t}{\e} \right) \exp\left( -\frac{ \ztmin t}{\e} \right) \; ,
\end{aligned}
$$  
which integrated in time gives 
$
\left| I_{2,2} \right| \sim O(\e).
$
On the other hand, by standard arguments of weak convergence, one easily proves thanks to the energy estimates that
$$
\int_{Q_T} \nabla \zeps \cdot \nabla \varphi \; d\bfx \; dt \to \int_{Q_T} \nabla \zz \cdot \nabla \varphi \;  d\bfx \; dt.
$$
The weak formulation \eqref{eq.weak.eps.time} tends, as $\e$ goes to zero, to
$$
\int_{\qt} \mu_{1,0} \dt \zz v(\bfx) w(t) d\bfx dt + \int_{\qt} \nabla \zz \cdot \nabla v(\bfx) w(t) d\bfx dt =0
$$
for every $v\in H^1_0(\om)\cap L^\infty(\om)$ and every $w \in L^\infty((0,T))$.
Thanks to Corollary \ref{coro.aubin} and Theorem \ref{thm.dt.zeps.l2}, $\zz \in C([0,T];H^1_0(\om))$
and $\dt \zz \in L^2(\qt)$. 
For the consistency with the initial condition it follows from Lemma \ref{lem.ci.veps.stab} in $L^1(\om)$.
Using the variational form \eqref{eq.weak.eps} at $t=0$, one obtains as well that
$$
\nrm{\zeps(\cdot,0)-\zp(\cdot,0)}{L^2(\om)} \lesssim o_\e(1)
$$
thanks to the fact that $\zp \in H^1_0(\om)$.
We consider a test function $\varphi \in \D(\qt)$. 
By Theorem III p.108 \cite{Schw.book}, 
the subspace of functions $\varphi(\bfx,t)$ of the form
$\varphi  := \sum_j v_j(\bfx) w_j(t)$ is dense in $\D(\qt)$.
Thus, the previous expression becomes~: for all $\varphi\in\D(\qt)$,
$$
< \mu_{1,0} \dt \zz,\varphi>_{\D'(\qt),\D(\qt)} = < \Delta \zz,\varphi>_{\D'(\qt),\D(\qt)}
$$
which means that (i) the equality holds a.e. in $\qt$ and  (ii) 
as $\mu_{1,0}\dt \zz \in L^2(\qt)$, so does $\Delta \zz$.

Using a test function $\varphi \in C^\infty([0,T]\times\om)$ vanishing on $[0,T]\times\partial \Omega$, 
for every fixed $t \in [0,T]$, one can
test the weak form \eqref{eq.weak.eps} and integrate in time, which implies that $\zeps$
solves~:
$$
\int_{\qt} \int_{\rr} \rhoe \veps da \varphi(\bfx,t) d\bfx dt + \int_{\qt} \nabla \zeps \cdot \nabla \varphi d\bfx dt =0 \;.
$$
This converges in the same way as above to the limit weak form \eqref{eq.f.v.z.0}
for every  test function $\varphi \in C^\infty([0,T]\times\om)$ vanishing on $[0,T]\times\partial \Omega$.
Now thanks to Lemma \ref{lem.density} this set is dense in $\dot{H}^1(\qt)$.
The integration by parts in time is well defined and gives~:
$$
\begin{aligned}
\left(\zz(\cdot,T)\mu_{1,0}(\cdot,T), \varphi(\cdot,T) \right) & + \int_0^T(\zz(\cdot,t), \dt  \mu_{1,0} \; \varphi +\mu_{1,0} \dt \varphi ) dt - \int_0^T ( \nabla \zz , \nabla \varphi ) dt \\
& = (\mu_{1,0}(\cdot,0) \zp(\cdot,0) ,\varphi(\cdot,0)) 
\end{aligned}
$$
for any $\varphi$ in $\dot{H}^1(\qt)$.  Thus $\zz$ is  a weak solution 
in the sense of \cite{lady.book} p.136. 
\end{proof}

\section{Adding a source term}\label{sec.b}

If one adds a source term to \eqref{eq.z.eps}, it becomes 
\begin{equation}
\label{eq.z.eps.source} 
\left\{ 
\begin{aligned} 
& \cL_\e(\zeps,\rhoe )  = \Delta_{\bfx} \zeps  + \cS(\bfx,t) \; , \quad & t\geq 0,\; \bfx \in \Omega \; , \\
& \zeps(\bfx,t) = 0 ,& \;t \in \rr  \; ,\; \bfx \in \dom ,\\
& \zeps(\bfx,t)=\zp(\bfx,t) \; , \quad  &t < 0 \; ,\; \bfx \in \Omega ,
\end{aligned}  
\right. 
\end{equation}
where we choose $\cS\in W^{1,\infty}((0,T);L^2(\om))$ for instance. We give some hints in order to 
extend the previous results. Existence and uniqueness for $\e$ fixed work the
same, we detail those of Section \ref{sec.nrj}.
The extension of Theorem \ref{thm.nrj} reads~:
\begin{thm}
 If $\cS \in W^{1,\infty}((0,T);L^2(\om))$ and under hypotheses \ref{hypo.data}, \ref{hypo.data.deux} and \ref{hypo.data.trois}, 
 one has that 
 $$
 \cet(\zeps(\cdot,t)) \leq C_1 \exp( C_2 t ) \cE_0(\zeps(\cdot,0)) + C_3
 $$
 and $\int_{(0,T)\times\rr \times\om} \zteps \rhoe \veps^2 da d\bfx dt < C_4$ as well.
 The constants $(C_i)_{i\in\{1,\dots,4\}}$ are independent on $\e$.
\end{thm}

\begin{proof}
 By similar arguments as in Theorem \ref{thm.nrj} we obtain~:
$$
\ddt{} \cet(\zeps(\cdot,t)) + \int_{\om\times\rr} \zteps \rhoe \veps^2 da d\bfx = \int_\om \dt \zeps \cS d\bfx 
$$
Thanks to Theorem \ref{thm.dtz.non.unif}, we can integrate by parts in time the latter expression
which gives~:
$$
\cet(\zeps(\cdot,t))\leq \cE_0(\zeps(\cdot,0))+ (\zeps(\cdot,t), \cS(\cdot,t))-(\zeps(\cdot,0), \cS(\cdot,0)) + \int_0^t (\zeps(\cdot,s), \dt \cS(\cdot,s)) ds
$$
where the parentheses denote the scalar product in $L^2(\om)$.
Then using Poincar\'e-Wirtinger in order to estimate $\nrm{\zeps(\cdot,t)}{L^2(\om)} \lesssim \cet(\zeps(\cdot,t))$ and Young's inequality twice (for a given positive $\delta$), one obtains~:
$$
\cet(\zeps(\cdot,t))\lesssim \delta \cet(\zeps(\cdot,t)) + \delta \int_0^t \cet(\zeps(\cdot,t)) + C
$$
where $C$ depends on $\cE_0(\zeps(\cdot,0))$ and on $\cS$. By Gronwall, one concludes.
\end{proof}

Of course since $\zeps$ now solves \eqref{eq.z.eps.source} the corresponding energy functional is 
to be redefined as 
$$
 \ti{\cal E}_t(w(\cdot)) :=\nud  \int_{\Omega} \left\{  | \nabla w |^2 + \int_{\rr} \frac{| w(\bfx)- \zeps(\bfx,t-\e a) |^2 }{\e} \rhoe(\bfx,t,a) da - \cS(\bfx,t)w(\bfx) \right\} d\bfx 
$$
\begin{lem}
 Under the same hypotheses as above, one has 
 $ \cet(\zeps(\cdot,0)) < C$.
\end{lem}

The rest follows the same lines as in the homogeneous case since 
the source term $\cS$ belongs to the appropriate functional space.

\section{The fully coupled problem}\label{sec.full}
For sake of simplicity we restrict ourselves in this section to the one-dimensional case
in the space variable, and set $\om := (0,1)$.
We consider here the case where $\zteps$ depends on the elongation.
The density of bonds  $\rhoe$ solves
\begin{equation}\label{eq.rho.eps.nl}
\left\{
 \begin{aligned}
 & \e \dt \rhoe + \da \rhoe + \zeta(\veps) \rhoe = 0 ,& (\xat) \in \om\times\rr\times(0,T), \\
 & \rhoe(\xat) = \beps(\bfx,t) ( 1 -\muze(\bfx,t) ), & (\xat) \in \om\times\{0\}\times(0,T), \\
 & \rhoe(\bfx,a,0) = \rhoi(\bfx,a),& (\bfx,a) \in \om \times\rr,
\end{aligned}
\right.
\end{equation}
coupled with the system 
\begin{equation}\label{eq.u.eps.nl}
\left\{
 \begin{aligned}
&  \e \dt \veps + \da \veps = g(\bfx,t), & (\xat) \in  \om\times\rr\times(0,T), \\
 & \veps(\bfx,0,t)=0,&  (\xat) \in \om\times\{0\}\times(0,T), \\
 &\veps(\bfx,a,t) =0, &  (\bfx,a,t) \in \dom\times \rr \times (0,T), \\
 & \veps (\bfx,a,0), = \vepsi(\bfx,a)& (\bfx,a) \in \om \times\rr,
\end{aligned}
\right.
\end{equation}
where $g$ solves in the variational sense in $\bhoz$~:
\begin{equation}\label{eq.g}
(\muze -\e  \Delta) g = \int_\rr \zteps \rhoe \veps da  + \e \dt \cS, \quad \alev \bfx \in \om 
\end{equation}
and we assume
\begin{hypo}\label{hypo.data.nl}
 Hypotheses \ref{hypo.data.deux} hold, moreover we add, 
  \begin{enumerate}[i)]
  \item $\zeta$ is a Lipschitz function s.t. $| \zeta'(u) |\leq \ztlip$ for all $u \in \RR$ and $\zeta(u)\geq \ztmin>0$ 
but there is not necessarily an upper bound
  \item $\beps$ is a given bounded function in space and time, moreover
$$
0 \leq \beps (\bfx,t) \leq \bmax \quad \alev (\bfx,t) \in \om\times(0,T).
$$
  \item for sake of simplicity we assume that $\cS\in W^{1,\infty}((0,T);L^2(\om))$,
\item for  $\zp$  the Lipschitz constant $C_{\zp} \in L^\infty_\bfx$. This implies that $\veps$ defined as in \eqref{eq.v.eps} satisfies $\vepsi/(1+a) \in  L^\infty_{\bfx,a}$.
\item  $\zeps(\bfx,0)$ satisfies the variational problem : find $ z \in \bhoz$ s.t.
  \begin{equation}
    \label{eq.zeps.t.0.nl}
(    \mu_{0,I} - \e \Delta ) z = \int_\rr \zp(\bfx,-\e a) \rhoi(\bfx,a) da + \cS(\bfx,0)
  \end{equation}
  \end{enumerate}
\end{hypo}
\newcommand{\YT}{Y_T}
We define the Banach space $\YT$
$$
\YT := \left\{ u \in \D'(\om\times\rr\times(0,T)) \st \frac{u}{1+a} \in L^\infty_{\bfx,a,t}\right\}. 
$$
endowed with its natural norm $\nrm{u}{\YT}:=\nrm{u/(1+a)}{L^\infty_{\bfx,a,t}}$.
\subsection{Existence and uniqueness for a truncated problem}
\newcommand{\uw}{u_w}
\begin{thm}
Under assumptions \ref{hypo.data.nl}, there is a unique solution
 $(\rho,w)\in L^\infty_tL^1_aL^\infty_\bfx\times Y_T$ solving (\ref{eq.rho.eps.nl}-\ref{eq.u.eps.nl}) where in the latter equation the right hand side  
 is replaced by $T_k(g_w)$, $T_k(g)$ being the usual truncation operator defined as 
 $T_k(g):=\max((-k),\min(g,k))$ for a fixed positive integer $k$ and $g_w$ solves \eqref{eq.g}.
\end{thm}
\newcommand{\hg}{\hat{g}}
\begin{proof}
 We proceed as in Theorem 3.2 \cite{MiOel.3}. Indeed,  for a given 
 $w$, $\rhow$ solves \eqref{eq.rho.eps.nl} with $\zeta(w)$ as the death rate. 
 The density $\rhow$ exists in the sense of characteristics and is unique in $L^\infty_t L^1_aL^\infty_\bfx$ as showed 
 above.
 One then computes \eqref{eq.g} with at the right hand side  $\int_\rr \zeta(w)\rhow w da$ as a first term.
 Then $\uw$ solves \eqref{eq.u.eps.nl}
 with the truncated right hand side  $T_k(g)$.  Since $|T_k(g)|\leq k$, if $w\in\YT$ so is $\uw$ invariably,  since
 $$
 \nrm{\uw}{\YT} \leq k
+ \nrm{\vepsi/(1+a)}{L^\infty_{\bfx}}.
 $$
 At this stage the map $\ov{\Phi}$ is complete $\uw = \ov{\Phi}(w)$ and   $\ov{\Phi}$ is endomorphic.
 Next we prove it is a contraction. 
 $$
 \left| \hg(\bfx,t) \right| \leq \nrm{\hg}{L^\infty_{\bfx}} \leq \omega \nrm{\hg}{\bhoz} 
 \lesssim 
 \nrm{\widehat{\int_{\rr} \zeta(w) \rhow w da}+ \hatmu g_1 }{L^1_{\bfx}} \leq C \nrm{\hatw}{\YT}
 $$ 
where $\hg := g_{w_2} - g_{w_1}$ and so on.  
The second estimate is due to the Sobolev embedding $\bhoz \subset C(\om)$ holding when $n=1$, 
while the third one is the consequence of the Lax-Milgram Theorem.
In order to obtain the last estimate above, we follow   the  steps  in part b) of the proof of   Theorem 3.2 \cite{MiOel.3}, 
the constraction follows
up to a time $T$ small enough. Then it is possbile to show (see part c) in the proof of Theorem 3.1 \cite{MiOel.3})
that the contraction time does not depend on the initial data but on $k$, 
 one concludes the global existence result.
\end{proof}
Theorem \ref{thm.stab.rho.veps} holds as well for $(\rhow,w)$, the solution of the truncated problem.

\subsection{A stability result}
Here the scope is to prove that the truncation constant $k$ can be chosen s.t.
$g$ solving \eqref{eq.g} actually never reaches truncation bounds $\{-k\}\cup\{k\}$. 
\begin{prop}\label{prop.borne.zt.veps.rhoe}
 Under assumptions \ref{hypo.data.nl}, let 
 $(\rhow,w)$ be the solution of the fully coupled and stabilized problem  \eqref{eq.rho.eps.nl}-\eqref{eq.u.eps.nl}-\eqref{eq.g},
with the modified source term $T_k(g)$ in \eqref{eq.u.eps.nl}, there exists a positive finite constant $\gamma_2$ s.t. 
 $$
p(t):= \int_{\rr\times\om} \zteps(w(t,a)) |w(t,a) | \rhow(t,a) d\bfx da \leq \gamma_2 ,\quad \forall t \geq 0 \; ,
 $$ 
 where the constant $\gamma_2$ depends on  the {\em a priori} bound on $\int_{\rr} \rhow |w| da$ (obtained in Theorem \ref{thm.stab.rho.veps}) ,
 $\nrm{\dt f}{L^\infty_tL^2_\bfx}$, $\ztlip$, and  $\zeta(0)$,
 but not on $k$.
\end{prop}

\begin{proof}
 Using equations \eqref{eq.rho.eps.nl}, \eqref{eq.u.eps.nl} and hypotheses on $\zeta$, one has
$$
\e \dt (\rhow | w | \zteps )  + \da (\rhow | w | \zteps ) +  \zteps^2 | w | \rhow \leq 
 \rhow |w| ( \e \dt \zteps + \da \zteps )  
 +  \zteps \rhow | T_k(g) | \; .
$$
 Integrating in age and space gives 
 \begin{equation}\label{eq.p}
 \begin{aligned}
   \e \dt p  & + \int_{\rr\times\om} \zteps^2   | w(t,a) | \rhow(t,a)  d \bfx da  \leq  \nrm{T_k(g)}{L^\infty_\bfx} 
\int_{\rr\times\om} ( \ztlip \rhow |w| + \zeta(w) \rhow(t,a) ) d \bfx da  \\
  & \leq \nrm{T_k(g)}{L^\infty_\bfx}  \left(  2 \ztlip\int_{\rr} \rhow |w| \, da  + \zeta(0) \right)   \leq ( 2 \ztlip / \gamma_1 + \zeta(0))
  \nrm{T_k(g)}{L^\infty_\bfx} , 
\end{aligned}
\end{equation}
where we use Theorem \ref{thm.stab.rho.veps} to bound the braquets on right hand side, and we denote 
$$
\int_{\rr\times\om} \rhow |w| d\bfx da \leq \gamma_1^{-1}.
$$
As the domain is one-dimensional, one has the embedding $\bhoz \subset C(\om)$ and 
thus there exists a constant $\omega$ independent of $g$ s.t. 
$$
\begin{aligned}
 \nrm{T_k(g(\cdot,t))}{L^\infty_\bfx} & \leq 
  \nrm{g(\cdot,t)}{L^\infty_\bfx} \leq \omega \nrm{g(\cdot,t)}{\bhoz} \leq  \frac{\omega}{\e} \nrm{ \int_\rr \zeta(w)\rhow w da+ \e \dt \cS }{H^{-1}_\bfx(\om)}\\
  & \leq \frac{\omega}{\e}  \nrm{ \int_\rr \zeta(w)\rhow w da  }{L^1_\bfx}+  \omega \nrm{ \dt \cS }{H^{-1}_\bfx} 
  \leq \frac{\omega}{\e} p + \omega \nrm{ \dt \cS }{H^{-1}_\bfx} 
\end{aligned}
  $$
Now we consider the second term in the left hand of \eqref{eq.p}~: using Jensen's inequality one writes
$$
\left( \frac{ \int_{\rr\times\om} \zeta(w) | w(\bfx,a,t) | \rhow(\bfx,a,t) d\bfx da}{\int_{\rr\times\om} | w | \rhow d\bfx da }\right)^2 
\leq \frac{ \int_{\rr\times\om} (\zteps(w))^2 | w(\bfx,a,t) | \rhow(\xat) d\bfx da}{\int_{\rr\times\om} | w | \rhow d\bfx da} \;,
$$
since $| w| \rhow / \int_{\rr\times\om} |w| \rhow \, d\bfx da $ is a unit measure on $\om\times\rr$. 
This implies that
$$
\int_{\om\times\rr} (\zteps(w))^2 | w(\xat) | \rhow(t,a) \; da d\bfx
\geq \frac{ \left( \int_{\om\times\rr} \zeta(w) | w(\xat) | \rhow(\xat) da d\bfx \right)^2 }{\int_{\om\times\rr}  |w| \rhow da d\bfx} 
\geq \gamma_1 p^2 \; .
$$
We obtain a Riccati inequality
$$
\e \dt p + \gamma_1 p^2 \leq h +   \omega p/\e  \;,\quad p(0) = \int_{\rr} \zteps(\vepsi(a))|\vepsi(a)|\rhoi(a)\; da \;, 
$$
where $h:=  \omega \nrm{\dt \cS}{L^\infty_t H^{-1}_\bfx} \left(  2 \ztlip / \gamma_1  + \zeta(0) \right)$ is a constant. 
We denote by $P_\pm$ the  solutions 
of the steady state equation associated to the last inequality, {\em i.e.} $P$ solves  $\gamma_1 P^2 -  P/\e -h= 0$. 
The solutions are given by 
$$
P_\pm =\left(  \frac{\omega}{\e} \pm \sqrt{\frac{\omega^2}{\e^2}  + 4 h  \gamma_1}\right)/(2\gamma_1) 
\leq\max\left( p(0), \left(  \omega + \sqrt{\omega^2  + 4 h  \gamma_1 \e^2 }\right)/(2 \e \gamma_1)\right)  =: \gamma_2.
$$
 Applying A.1 in the appendix \cite{MiOel.3},
we conclude that $p(t) \leq \max\{p(0), P_+\}\leq \gamma_2$,
which ends the proof.
\end{proof}
\begin{coro}\label{coro.ex.glob}
 Under hypotheses \ref{hypo.data.nl}, there exists a unique global solution of (\ref{eq.rho.eps.nl}-\ref{eq.u.eps.nl}-\ref{eq.g}).
\end{coro}
\begin{proof}
 It suffices to take $k > \gamma_2/\e + \nrm{\dt \cS}{L^\infty_t H^{-1}_\bfx}$, and it is clear from above
 that $g$ never reaches $k$ a.e. $\bfx,t$. Thus the solution $(\rhow,w)$ is also the solution of
 (\ref{eq.rho.eps.nl}-\ref{eq.u.eps.nl}-\ref{eq.g}) without truncating $g$. Thus existence is proved.
 Since the truncated solution pair is unique so is the latter one.
\end{proof}

\subsection{If $\bmin>0$, the total bonds' population  never vanishes}

Once a global $L^\infty_{\xat}$ bound is proved for $g$ solving \eqref{eq.g}, one should use again arguments 
of Lemma 4.1, Proposition 4.2 and 4.3 \cite{MiOel.3} and prove exactly in the same way~:
\begin{thm}\label{thm.positive.for.ever}
 Under assumptions \ref{hypo.data.nl}, if $\beps\geq\bmin>0$ and $\nrm{\mu_{0,I}}{L^\infty_\bfx} \leq \gamma_0 < 1$, then
the solution $(\rhoe,\veps)$ of (\ref{eq.rho.eps.nl}-\ref{eq.u.eps.nl}-\ref{eq.g}) satisfies 
\begin{enumerate}[i)]
 \item defining $\gamma_1>0$ as $\gamma_1 <\min\left( 1-\gamma_0, \ztmin/(\ztmin + \bmax)\right)$
 $$
 \nrm{\muze}{L^\infty_\bfx} \leq 1- \gamma_1,\quad \forall t>0,
 $$
 \item this in turn implies that
 $$
 \frac{ \int_\rr \zeta(\veps(\xat))\rhoe(\xat)da }{\int_\rr \rhoe(\xat)da} \leq \zeta(0) + C \left( 1 + \nrm{\frac{\vepsi}{1+a}}{L^\infty_{\bfx,a}} \right)\frac{2 }{\gamma_1 \bmin} \nrm{g}{L^\infty_{\bfx,t}} =: \gamma_2
 $$
 for almost every $\bfx $ in $ \om$.
 \item choosing $\mumin>0$ s.t.
 $$
 \mumin < \min\left( \inf_{\bfx \in \om} \mu_{0,I}(\bfx), \frac{\bmin}{\bmin+\gamma_2} \right)
 $$
 one guarantees that
 $$
 \muze (\bfx,t) \geq \mumin,\quad \alev \bfx \in \om, \quad \forall t >0.
 $$
\end{enumerate}
\end{thm}
This result proves that it is not possible to have extinction of bonds at the contrary to the situation 
observed in \cite{MiOel.3}.

\subsection{Equivalence with the initial formulation}

\begin{lem}
Under  hypotheses \ref{hypo.data.nl}, if $(\rhoe,\veps)$ solves (\ref{eq.rho.eps.nl}-\ref{eq.u.eps.nl}-\ref{eq.g}),
then defining 
$
\zeps(\bfx,t):=\int_0^t g(\bfx,s) ds + \zeps(\bfx,0)
$
where $\zeps(\bfx,0)$ is the solution of \eqref{eq.zeps.t.0.nl},
one has \eqref{eq.def.veps} and 
$(\rhoe,\zeps)\in C_tL^1_aL^\infty_\bfx\times C^1_t L^\infty_\bfx$ solves \eqref{eq.rho.eps.nl} coupled with \eqref{eq.z.eps}.
\end{lem}

\begin{proof}
Using the method of characteristics, starting from \eqref{eq.u.eps.nl}, one recovers
by definition of $\zeps$ \eqref{eq.def.veps}.
 Using \eqref{eq.u.eps.nl} and integrating against $\rhoe$, one has
 $$
 \begin{aligned}
&  \left( \e \dt \int_{\rr} \rhoe \veps da  + \int_\rr \zeta \rhoe \veps da ,v\right)= \\
& \quad= (\muze g,v) = - \e (\nabla g,\nabla v)  + \left( \int_{\rr} \zeta \rhoe \veps da + \e \dt \cS , v \right),\quad \forall v \in \bhoz,  
 \end{aligned}
 $$
where the exterior barquets denote the scalar product in $L^2_\bfx$.
  After a simplification and integration in time, the latter expression becomes :
 $$
\left(  \int_{\rr} \rhoe \veps da,v\right) =  - \left( \nabla \int_0^t g(\bfx,s) ds ,\nabla v \right)
+ \left( \cS(\bfx,t)- \cS(\bfx,0) + \int_\rr \rhoi \vepsi da ,v \right),
 $$
 but because of the definition of $\vepsi$ and $\zeps(\bfx,0)$ one recovers that
  $$
  \begin{aligned}
 \left( \int_{\rr} \rhoe \veps da,v\right) & = - \left( \nabla \left( \int_0^t g(\bfx,s) ds + \zeps(\bfx,0)\right),\nabla v\right) + (\cS(\bfx,t),v) \\
& = -(\nabla \zeps(\bfx,t),\nabla v) + (\cS(\bfx,t),v),    
  \end{aligned}
 $$
 which ends the proof.
\end{proof}

\subsection{Positivity and concluding remarks}

\begin{thm}\label{thm.positivity.veps}
Under assumptions \ref{hypo.data.nl}, if moreover  $\vepsi(a)\geq0$ for a.e. $(\bfx,a)\in \rr$ and  $\dt \cS(\bfx,t) \geq 0$ for a.e. $(\bfx,t)\in\om\times (0,T)$, 
then $ \veps(\xat)$ 
is non-negative for a.e. $(\xat) \in \om\times(\rr)^2$.
\end{thm}

\begin{proof}{}
We define $[u]_-$ (resp. $[u]_+$) the negative (resp. positive) part of $u$ {\rm i.e.} $[u]_- := \min(0,u)$,
(resp. $[u]_+ := \max(0,u)$). We set $H_-(u):=-\sgn_-(u)$ with $\sgn_-$ being the negative part of the sign function.
We look for $(\rhoe,\vepsp)$ solving  the coupled system :
\begin{equation}\label{eq.veps.mod.plus}
  \left\{ 
\begin{aligned}
& \e \dt \vepsp + \da \vepsp = g_+ &\bfx \in \om,\, a>0,  t>0 \, , \\
& (\muze-\e \Delta)g_+= \left(  \e \dt \cS + \int_0^\infty \left(\zeta_\e [\vepsp]_+ \rhoe\right)(t,\tia) \; d\tia \right) \,,& (\bfx,t) \in \om\times (0,T),\\
& \vepsp(\bfx,0,t)= 0\,, &t>0  \; , \\
& \vepsp(\bfx,a,t)= 0\,, &\bfx\in \dom,\; t>0  \; , \\
& \vepsp(\bfx,a,0)=\vepsi(\bfx,a)\,, &\bfx \in \om,\, a > 0  \; ,
\end{aligned}
\right.
\end{equation}
together with $\rhoe(\vepsp)$ being the solution of \eqref{eq.rho.eps.nl} with the death rate $\zeta(\vepsp)$. 
The results of Corollary \ref{coro.ex.glob}  can be repeated and provide global existence and uniqueness. 
Multiplying \eqref{eq.veps.mod.plus} by $H_-(\vepsp)$, (cf the rigorous explanation in Lemma 3.1 \cite{MiOel.1} that holds here for a. e. $\bfx \in \om$), one gets 
$$
\e \dt [\vepsp]_- + \da [\vepsp]_- = H_-(\vepsp) g_+ \;. 
$$
Because of the weak maximum principle (Theorem 8.1, p.179 \cite{GiTru.Book}), $g_+ \geq 0$.
As  $H_-$ is positive, one concludes that~:
$$
\e \dt [\vepsp]_- + \da [\vepsp]_- \geq 0
$$
which using the Duhamel formula provides that
$$
0 \geq [\vepsp(\bfx,a,t)]_- \geq 
\begin{cases}
 [\vepsp(\bfx,0,t-\e a)]_- = 0 &  \text{ if }t\geq \e a, \\
  [\vepsi( \bfx,a - t / \e )]_- = 0 &  \text{ if }t\leq \e a. \\
\end{cases}
$$
for almost every $\bfx\in\om$.
But as $\vepsp$ is then almost everywhere positive $(\rhoe,\vepsp)$ satisfies as well 
system \eqref{eq.rho.eps.nl}-\eqref{eq.u.eps.nl}, which by uniqueness
 proves that actually $(\rhoe(\veps),\veps)=(\rhoe(\vepsp),\vepsp)$, which implies
the claim.
\end{proof}

Under hypotheses above $\veps$ is  positive and thus the equation satisfied bu $\muze$
can be made explicit if we suppose that $\zeta(u):=1+|u|$ for instance. Indeed, 
$$
\begin{aligned}
& \e \dt \muze - \beps ( 1 - \muze ) + \int_\rr \zeta(\veps) \rhoe da  = \e \dt \muze - \beps ( 1 - \muze ) + \muze + \int_\rr \rhoe \veps da =\\
&= \e \dt \muze - \beps ( 1 - \muze ) + \muze +  \cS  + \Delta \zeps = 0. \\
\end{aligned}
$$
According to that, one sees that there is a new balance of terms when compared 
to the case without the Laplace operator considered in \cite{MiOel.3}.
\begin{equation}\label{eq.mu.nl}
 (\e \dt + (\beps +1)) \muze + \Delta \zeps + \cS = \beps
\end{equation}
Indeed without the Laplace operator, there could be a sufficient tear-off ($\cS$ large enough)
so that the birth source term becomes too small and $\muze$ is shown to go to zero 
in finite time (see Proposition 7.3 \cite{MiOel.3}). 
It suffices to take $\cS_{\min} > \bmax$ for example.
Here instead, the presence of the 
Laplace operator stabilizes the exterior force and provides 
global existence.  One observes that if $\cS$ and $\beps$  converge as time grows to some functions of $\bfx$, 
the asymptotic profile (for large times) $(\rho_\infty,z_\infty)$ is s.t.  
 $$
 \cS_\infty = -\Delta z_\infty, \quad \mu_\infty(\bfx) = \frac{\beta_\infty}{\beta_\infty +1}.
 $$
Indeed $\cL_\e z_\infty=0$ so that the first equation is the asymptotic limit in time of \eqref{eq.z.eps},
while the second comes from \eqref{eq.mu.nl} with $\dt \mu_\infty=0$.

\subsection{A numerical simulation}

We discretize \eqref{eq.rho.eps.nl} using an explicit upwind method with the CFL constant 
being equal to 1. We use a trapezoidal rule to compute the non-local boundary condition $\rhoe(\bfx,0,t)=\beps(\bfx,t)(1-\muze)$.
We solve \eqref{eq.z.eps} using a P2 Discontinuous Galerkin method
for the Laplace operator 
in space \cite{ErGu.04.book} and a trapezoidal rule to dicretize $\cL_\e$.

The constants are defined as : $\cS=1e4$, $\zp(\bfx,t)=\sin(\pi \bfx)/\pi$, the initial condition 
for $\rhoi = \exp(- a)$ is uniform with respect $\bfx$, $\zeta(u)=1+|u|$ and the maximal age
is $a_{\max}=10$ with a discretisation step $\Delta a=10^4$ and $\e=1e-3$. 
The on-rate $\beta$ in \eqref{eq.rho.eps.nl}  is defined s.t. it is $\zeps$ dependent 
$$
\beta(\bfx,t) = 
\begin{cases}
1 & \text{ if } \zeps \in (0,\ov{z}) \\
0 & \text{ otherwise}
\end{cases}
$$
with $\ov{z} = 1000$, 
then we observe at least locally in space that total extinction of bonds' population occurs.

\begin{figure}[ht!]
\begin{minipage}[c]{.46\linewidth}
\begingroup
  \makeatletter
  \providecommand\color[2][]{%
    \GenericError{(gnuplot) \space\space\space\@spaces}{%
      Package color not loaded in conjunction with
      terminal option `colourtext'%
    }{See the gnuplot documentation for explanation.%
    }{Either use 'blacktext' in gnuplot or load the package
      color.sty in LaTeX.}%
    \renewcommand\color[2][]{}%
  }%
  \providecommand\includegraphics[2][]{%
    \GenericError{(gnuplot) \space\space\space\@spaces}{%
      Package graphicx or graphics not loaded%
    }{See the gnuplot documentation for explanation.%
    }{The gnuplot epslatex terminal needs graphicx.sty or graphics.sty.}%
    \renewcommand\includegraphics[2][]{}%
  }%
  \providecommand\rotatebox[2]{#2}%
  \@ifundefined{ifGPcolor}{%
    \newif\ifGPcolor
    \GPcolortrue
  }{}%
  \@ifundefined{ifGPblacktext}{%
    \newif\ifGPblacktext
    \GPblacktexttrue
  }{}%
  \let\gplgaddtomacro\g@addto@macro
  \gdef\gplbacktext{}%
  \gdef\gplfronttext{}%
  \makeatother
  \ifGPblacktext
    \def\colorrgb#1{}%
    \def\colorgray#1{}%
  \else
    \ifGPcolor
      \def\colorrgb#1{\color[rgb]{#1}}%
      \def\colorgray#1{\color[gray]{#1}}%
      \expandafter\def\csname LTw\endcsname{\color{white}}%
      \expandafter\def\csname LTb\endcsname{\color{black}}%
      \expandafter\def\csname LTa\endcsname{\color{black}}%
      \expandafter\def\csname LT0\endcsname{\color[rgb]{1,0,0}}%
      \expandafter\def\csname LT1\endcsname{\color[rgb]{0,1,0}}%
      \expandafter\def\csname LT2\endcsname{\color[rgb]{0,0,1}}%
      \expandafter\def\csname LT3\endcsname{\color[rgb]{1,0,1}}%
      \expandafter\def\csname LT4\endcsname{\color[rgb]{0,1,1}}%
      \expandafter\def\csname LT5\endcsname{\color[rgb]{1,1,0}}%
      \expandafter\def\csname LT6\endcsname{\color[rgb]{0,0,0}}%
      \expandafter\def\csname LT7\endcsname{\color[rgb]{1,0.3,0}}%
      \expandafter\def\csname LT8\endcsname{\color[rgb]{0.5,0.5,0.5}}%
    \else
      \def\colorrgb#1{\color{black}}%
      \def\colorgray#1{\color[gray]{#1}}%
      \expandafter\def\csname LTw\endcsname{\color{white}}%
      \expandafter\def\csname LTb\endcsname{\color{black}}%
      \expandafter\def\csname LTa\endcsname{\color{black}}%
      \expandafter\def\csname LT0\endcsname{\color{black}}%
      \expandafter\def\csname LT1\endcsname{\color{black}}%
      \expandafter\def\csname LT2\endcsname{\color{black}}%
      \expandafter\def\csname LT3\endcsname{\color{black}}%
      \expandafter\def\csname LT4\endcsname{\color{black}}%
      \expandafter\def\csname LT5\endcsname{\color{black}}%
      \expandafter\def\csname LT6\endcsname{\color{black}}%
      \expandafter\def\csname LT7\endcsname{\color{black}}%
      \expandafter\def\csname LT8\endcsname{\color{black}}%
    \fi
  \fi
    \setlength{\unitlength}{0.0500bp}%
    \ifx\gptboxheight\undefined%
      \newlength{\gptboxheight}%
      \newlength{\gptboxwidth}%
      \newsavebox{\gptboxtext}%
    \fi%
    \setlength{\fboxrule}{0.5pt}%
    \setlength{\fboxsep}{1pt}%
\begin{picture}(3960.00,2820.00)%
    \gplgaddtomacro\gplbacktext{%
      \csname LTb\endcsname%
      \put(648,508){\makebox(0,0)[r]{\strut{}$0$}}%
      \csname LTb\endcsname%
      \put(648,813){\makebox(0,0)[r]{\strut{}$200$}}%
      \csname LTb\endcsname%
      \put(648,1117){\makebox(0,0)[r]{\strut{}$400$}}%
      \csname LTb\endcsname%
      \put(648,1422){\makebox(0,0)[r]{\strut{}$600$}}%
      \csname LTb\endcsname%
      \put(648,1727){\makebox(0,0)[r]{\strut{}$800$}}%
      \csname LTb\endcsname%
      \put(648,2032){\makebox(0,0)[r]{\strut{}$1000$}}%
      \csname LTb\endcsname%
      \put(648,2336){\makebox(0,0)[r]{\strut{}$1200$}}%
      \csname LTb\endcsname%
      \put(648,2641){\makebox(0,0)[r]{\strut{}$1400$}}%
      \csname LTb\endcsname%
      \put(737,349){\makebox(0,0){\strut{}$0$}}%
      \csname LTb\endcsname%
      \put(1033,349){\makebox(0,0){\strut{}$0.1$}}%
      \csname LTb\endcsname%
      \put(1328,349){\makebox(0,0){\strut{}$0.2$}}%
      \csname LTb\endcsname%
      \put(1624,349){\makebox(0,0){\strut{}$0.3$}}%
      \csname LTb\endcsname%
      \put(1919,349){\makebox(0,0){\strut{}$0.4$}}%
      \csname LTb\endcsname%
      \put(2215,349){\makebox(0,0){\strut{}$0.5$}}%
      \csname LTb\endcsname%
      \put(2510,349){\makebox(0,0){\strut{}$0.6$}}%
      \csname LTb\endcsname%
      \put(2806,349){\makebox(0,0){\strut{}$0.7$}}%
      \csname LTb\endcsname%
      \put(3101,349){\makebox(0,0){\strut{}$0.8$}}%
      \csname LTb\endcsname%
      \put(3396,349){\makebox(0,0){\strut{}$0.9$}}%
      \csname LTb\endcsname%
      \put(3692,349){\makebox(0,0){\strut{}$1$}}%
      \csname LTb\endcsname%
      \put(973,1727){\makebox(0,0)[l]{\strut{}$\beta=1$}}%
      \csname LTb\endcsname%
      \put(2215,2489){\makebox(0,0)[l]{\strut{}$\beta=0$}}%
    }%
    \gplgaddtomacro\gplfronttext{%
      \csname LTb\endcsname%
      \put(123,1574){\rotatebox{-270}{\makebox(0,0){\strut{}$z$}}}%
      \csname LTb\endcsname%
      \put(2214,111){\makebox(0,0){\strut{}$\bfx$}}%
      \csname LTb\endcsname%
      \put(2183,1871){\makebox(0,0)[r]{\strut{}t=1e-4}}%
      \csname LTb\endcsname%
      \put(2183,1673){\makebox(0,0)[r]{\strut{}t=2e-4}}%
      \csname LTb\endcsname%
      \put(2183,1475){\makebox(0,0)[r]{\strut{}t=3e-4}}%
      \csname LTb\endcsname%
      \put(2183,1277){\makebox(0,0)[r]{\strut{}$\ov{z}$}}%
    }%
    \gplbacktext
    \put(0,0){\includegraphics{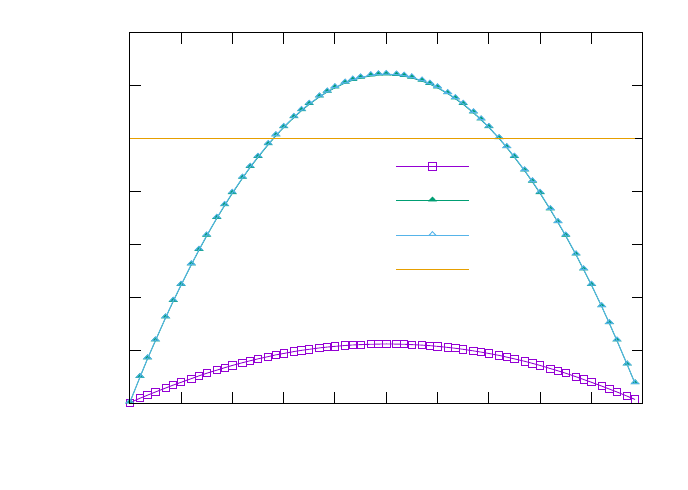}}%
    \gplfronttext
  \end{picture}%
\endgroup
\caption{ $\zeps(\bfx,t)$ at given times, when $t>2e-4$ the curves are superposed}\label{fig.z}
\end{minipage} \hfill
   \begin{minipage}[c]{.46\linewidth}
\begingroup
  \makeatletter
  \providecommand\color[2][]{%
    \GenericError{(gnuplot) \space\space\space\@spaces}{%
      Package color not loaded in conjunction with
      terminal option `colourtext'%
    }{See the gnuplot documentation for explanation.%
    }{Either use 'blacktext' in gnuplot or load the package
      color.sty in LaTeX.}%
    \renewcommand\color[2][]{}%
  }%
  \providecommand\includegraphics[2][]{%
    \GenericError{(gnuplot) \space\space\space\@spaces}{%
      Package graphicx or graphics not loaded%
    }{See the gnuplot documentation for explanation.%
    }{The gnuplot epslatex terminal needs graphicx.sty or graphics.sty.}%
    \renewcommand\includegraphics[2][]{}%
  }%
  \providecommand\rotatebox[2]{#2}%
  \@ifundefined{ifGPcolor}{%
    \newif\ifGPcolor
    \GPcolortrue
  }{}%
  \@ifundefined{ifGPblacktext}{%
    \newif\ifGPblacktext
    \GPblacktexttrue
  }{}%
  \let\gplgaddtomacro\g@addto@macro
  \gdef\gplbacktext{}%
  \gdef\gplfronttext{}%
  \makeatother
  \ifGPblacktext
    \def\colorrgb#1{}%
    \def\colorgray#1{}%
  \else
    \ifGPcolor
      \def\colorrgb#1{\color[rgb]{#1}}%
      \def\colorgray#1{\color[gray]{#1}}%
      \expandafter\def\csname LTw\endcsname{\color{white}}%
      \expandafter\def\csname LTb\endcsname{\color{black}}%
      \expandafter\def\csname LTa\endcsname{\color{black}}%
      \expandafter\def\csname LT0\endcsname{\color[rgb]{1,0,0}}%
      \expandafter\def\csname LT1\endcsname{\color[rgb]{0,1,0}}%
      \expandafter\def\csname LT2\endcsname{\color[rgb]{0,0,1}}%
      \expandafter\def\csname LT3\endcsname{\color[rgb]{1,0,1}}%
      \expandafter\def\csname LT4\endcsname{\color[rgb]{0,1,1}}%
      \expandafter\def\csname LT5\endcsname{\color[rgb]{1,1,0}}%
      \expandafter\def\csname LT6\endcsname{\color[rgb]{0,0,0}}%
      \expandafter\def\csname LT7\endcsname{\color[rgb]{1,0.3,0}}%
      \expandafter\def\csname LT8\endcsname{\color[rgb]{0.5,0.5,0.5}}%
    \else
      \def\colorrgb#1{\color{black}}%
      \def\colorgray#1{\color[gray]{#1}}%
      \expandafter\def\csname LTw\endcsname{\color{white}}%
      \expandafter\def\csname LTb\endcsname{\color{black}}%
      \expandafter\def\csname LTa\endcsname{\color{black}}%
      \expandafter\def\csname LT0\endcsname{\color{black}}%
      \expandafter\def\csname LT1\endcsname{\color{black}}%
      \expandafter\def\csname LT2\endcsname{\color{black}}%
      \expandafter\def\csname LT3\endcsname{\color{black}}%
      \expandafter\def\csname LT4\endcsname{\color{black}}%
      \expandafter\def\csname LT5\endcsname{\color{black}}%
      \expandafter\def\csname LT6\endcsname{\color{black}}%
      \expandafter\def\csname LT7\endcsname{\color{black}}%
      \expandafter\def\csname LT8\endcsname{\color{black}}%
    \fi
  \fi
    \setlength{\unitlength}{0.0500bp}%
    \ifx\gptboxheight\undefined%
      \newlength{\gptboxheight}%
      \newlength{\gptboxwidth}%
      \newsavebox{\gptboxtext}%
    \fi%
    \setlength{\fboxrule}{0.5pt}%
    \setlength{\fboxsep}{1pt}%
\begin{picture}(3960.00,2820.00)%
    \gplgaddtomacro\gplbacktext{%
      \csname LTb\endcsname%
      \put(648,508){\makebox(0,0)[r]{\strut{}$1e{-8}$}}%
      \csname LTb\endcsname%
      \put(648,775){\makebox(0,0)[r]{\strut{}$1e{-7}$}}%
      \csname LTb\endcsname%
      \put(648,1041){\makebox(0,0)[r]{\strut{}$1e{-6}$}}%
      \csname LTb\endcsname%
      \put(648,1308){\makebox(0,0)[r]{\strut{}$1e{-5}$}}%
      \csname LTb\endcsname%
      \put(648,1575){\makebox(0,0)[r]{\strut{}$1e{-4}$}}%
      \csname LTb\endcsname%
      \put(648,1841){\makebox(0,0)[r]{\strut{}$1e{-3}$}}%
      \csname LTb\endcsname%
      \put(648,2108){\makebox(0,0)[r]{\strut{}$1e{-2}$}}%
      \csname LTb\endcsname%
      \put(648,2374){\makebox(0,0)[r]{\strut{}$1e{-1}$}}%
      \csname LTb\endcsname%
      \put(648,2641){\makebox(0,0)[r]{\strut{}$1e{0}$}}%
      \csname LTb\endcsname%
      \put(737,349){\makebox(0,0){\strut{}$0$}}%
      \csname LTb\endcsname%
      \put(1033,349){\makebox(0,0){\strut{}$0.1$}}%
      \csname LTb\endcsname%
      \put(1328,349){\makebox(0,0){\strut{}$0.2$}}%
      \csname LTb\endcsname%
      \put(1624,349){\makebox(0,0){\strut{}$0.3$}}%
      \csname LTb\endcsname%
      \put(1919,349){\makebox(0,0){\strut{}$0.4$}}%
      \csname LTb\endcsname%
      \put(2215,349){\makebox(0,0){\strut{}$0.5$}}%
      \csname LTb\endcsname%
      \put(2510,349){\makebox(0,0){\strut{}$0.6$}}%
      \csname LTb\endcsname%
      \put(2806,349){\makebox(0,0){\strut{}$0.7$}}%
      \csname LTb\endcsname%
      \put(3101,349){\makebox(0,0){\strut{}$0.8$}}%
      \csname LTb\endcsname%
      \put(3396,349){\makebox(0,0){\strut{}$0.9$}}%
      \csname LTb\endcsname%
      \put(3692,349){\makebox(0,0){\strut{}$1$}}%
    }%
    \gplgaddtomacro\gplfronttext{%
      \csname LTb\endcsname%
      \put(123,1574){\rotatebox{-270}{\makebox(0,0){\strut{}$\mu$}}}%
      \csname LTb\endcsname%
      \put(2214,111){\makebox(0,0){\strut{}$\bfx$}}%
      \csname LTb\endcsname%
      \put(2183,1772){\makebox(0,0)[r]{\strut{}t=1e-4}}%
      \csname LTb\endcsname%
      \put(2183,1574){\makebox(0,0)[r]{\strut{}t=2e-4}}%
      \csname LTb\endcsname%
      \put(2183,1376){\makebox(0,0)[r]{\strut{}t=3e-4}}%
    }%
    \gplbacktext
    \put(0,0){\includegraphics{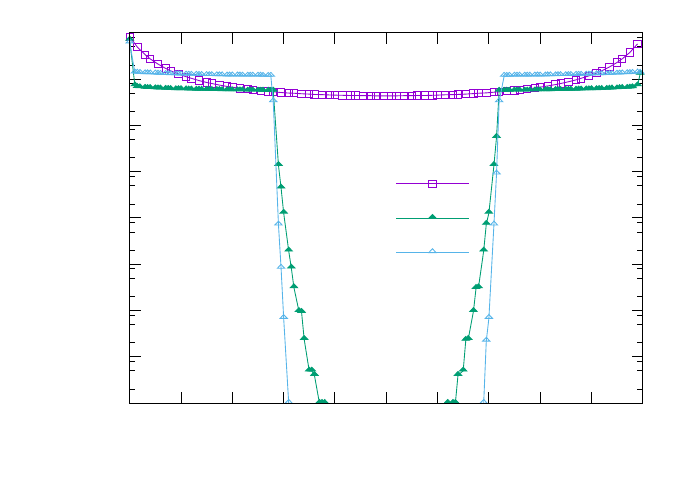}}%
    \gplfronttext
  \end{picture}%
\endgroup

\caption{ $\muze(\bfx,t)$ at given times, in logscale }\label{fig.mu}
\end{minipage} \hfill
\end{figure}

The non-local boundary condition does not exactly fit in the framework presented in this section,
since $\beps$ depends on $\zeps$ which is not in assumptions \eqref{hypo.data.nl}, but conditionally 
$\beps$ vanishes. Nevertheless the previous results could be extended in this case. 
The fact that $\beps$ vanishes contradicts the hypothesis of Theorem \ref{thm.positive.for.ever}. 
We display in figure \ref{fig.z} the displacement $\zeps$ as a function of $\bfx$ for  different 
times and in figure \ref{fig.mu},  $\muze$ is displayed as well. 
Since the convergence towards the steady state is exponential we focus on small times.
One observes that asymptotically in time  two regimes occur:
either $z_\infty > \ov{z}$ and then $\mu_\infty=0$, or $z_\infty<\ov{z}$ and $\mu_\infty=\nud$.
One should note that in this case, there is an elliptic-parabolic transition inside the
domain since $\cL_\e$ may vanish on some compact sub-interval.

In order to conclude, this simulation show that in order to have detachments of
an adhesion site, it seems that a necessary condition is that
the adhesion on-rate should vanish, 
at the contrary to what was shown in the single point adhesion model \cite{MiOel.3}
where the explosion of the non-linear death-rate was enough.

\appendix

\addtocontents{toc}{\protect\setcounter{tocdepth}{1}}
\makeatletter
\addtocontents{toc}{%
  \begingroup
  \let\protect\l@chapter\protect\l@section
  \let\protect\l@section\protect\l@subsection
}
\makeatother

 \section{Euler-Lagrange equation versus minimization}\label{sec.equiv}

 \begin{lem}\label{lem.equiv}
 The function $\zeps \in X_T$ is the weak solution of system \eqref{eq.z.eps} if and only if it satisfies \eqref{eq.minimiz}-\eqref{eq.nrj}.
 \end{lem}

 \begin{proof}
  As the square function is convex, one has for all $v \in H^1_0(\om)$ that~:
  $$
 \begin{aligned}
 \nud &\left\{ ( \zeps(\bfx,t)-\zeps(\bfx,t-\e a))^2 - 
 ( v(\bfx)-\zeps(\bfx,t-\e a))^2 \right\} \\
 &\leq 
 ( \zeps(\bfx,t)-\zeps(\bfx,t-\e a))(\zeps(\bfx,t)-v(\bfx))   
 \end{aligned}
  $$
 multiplying by $\rhoe\geq 0$,  integrating in age and then in space, one gets that
 $$
 \begin{aligned}
 \frac{1}{2 \e}&\left\{ \int_\om \int_\rr( \zeps(\bfx,t)-\zeps(\bfx,t-\e a))^2 \rhoe(\bfx,a,t) da d\bfx  \right. \\
 & - \left.\int_\om\int_\rr ( v(\bfx)-\zeps(\bfx,t-\e a))^2 \rhoe(\bfx,a,t) da d \bfx \right\} 
 \leq 
 \left( \cL_\e( \zeps,\rhoe ),\zeps-v\right) . 
 \end{aligned}
  $$
  As $\zeps$ is a weak solution in the sense of Definition \ref{def.weak.sol.z.eps}, 
  and $\zeps-v$ is in the test space, one can write that
  $$
  ( \cL_\e(\zeps,\rhoe ),\zeps-v) + (\nabla \zeps,\nabla(\zeps -v)) = 0
  $$
  and using the previous convexity argument, one concludes that
  $$
  \cet(\zeps(\cdot,t)) \leq \cet(v), \quad \forall v \in H^1_0(\om).
  $$
  Conversely, set $i(\tau):= \cet(\zeps+\tau v)$ for any $v \in H^1_0(\om)$, then since $\zeps$ satisfies 
  \eqref{eq.minimiz}, one has $i'(0)=0$. As the expression is explicit with respect to $\tau$ the claim follows by simple computations. 
 \end{proof}
\vspace{-0.75 cm}

\section{A density result}

\begin{lem}\label{lem.density}
 $\Omega$ is a Lipschitz bounded domain, the set  $\{ u \in C^\infty(\Omega\times[0,T])$ s.t. 
 $u=0$ on 
 $\partial \Omega\times [0,T]\}$, 
  is dense in $\dot{H}^1(\qt) := \{ u \in H^1(\Omega\times[0,T])$ s.t. $u =0 \text{ on }\partial \Omega\times [0,T]\}$
 endowed with the 
 $H^1(\Omega\times[0,T])$ norm.
\end{lem}

\begin{proof}
 According to \cite{lady.book}, p.89, Lemma 4.12, 
 the set of functions of the form $\sum_{k=1}^N d_k(t)\psi_k(x)$ is dense in $\dot{H}^1(\qt)$, 
 where $d_k(t) \in C^\infty([0,T])$ and $(\psi_k)_{k\in{\mathbb N}}$ is a fundamental system of functions in $H^1_0(\Omega)$.
Then, approximating each $\psi_k\in H^1_0(\Omega)$ by a ${\mathcal D}(\Omega)$ 
function completes the proof, since the number $N$ is finite. 
\end{proof}

\addtocontents{toc}{\endgroup}
\vspace{-0.75 cm}
\section*{References}

\bibliographystyle{alpha}
\def\cprime{$'$} \def\cprime{$'$}


\end{document}